\documentclass[]{amsart}

\usepackage[english]{babel}
\usepackage[T1]{fontenc}

\usepackage[a4paper]{geometry}
\usepackage{amsmath}
\usepackage{verbatim}
\usepackage[colorinlistoftodos]{todonotes}
\usepackage{mathtools} %for \xhookrightarrow

\usepackage[all]{xy}
\usepackage{amssymb,mathrsfs, amscd, graphicx, array, multirow,
enumerate, amsfonts, euscript}
\usepackage{comment}
\usepackage{upgreek} % upcase of Greek letters
\usepackage{float}
\usepackage{enumitem}

\usepackage{marvosym}
\xyoption{poly}
\usepackage{url}

\allowdisplaybreaks
\usepackage{color}
\usepackage{xcolor}

\usepackage{tikz-cd}
\usepackage{tikz}
\usetikzlibrary{matrix,arrows}
\theoremstyle{plain}
\newtheorem{thm}{Theorem}[section]
\newtheorem{lemma}[thm]{Lemma}
\newtheorem{lem}[thm]{Lemma}
\newtheorem{prop}[thm]{Proposition}
\newtheorem{cor}[thm]{Corollary}

\theoremstyle{definition}

\newtheorem{defn}[thm]{Definition}
\newtheorem{example}[thm]{Example}

\theoremstyle{remark}

\newtheorem{rem}[thm]{Remark}

\numberwithin{equation}{section}
\newcommand{\abs}[1]{\lvert #1 \rvert}

\def\Hom{\mathrm{Hom}}
\def\Div{\mathrm{Div}}

\def\makeop#1{\expandafter\def\csname#1\endcsname
  {\mathop{\rm #1}\nolimits}\ignorespaces}
\makeop{Hom}   \makeop{End}   \makeop{Aut}   \makeop{Isom}  \makeop{Pic} 
\makeop{Gal}   \makeop{ord}   \makeop{Char}  \makeop{Div}   \makeop{Lie} 
\makeop{PGL}   \makeop{Corr}  \makeop{PSL}   \makeop{sgn}   \makeop{Spf}
\makeop{Spec}  \makeop{Tr}    \makeop{Nr}    \makeop{Fr}    \makeop{disc}
\makeop{Proj}  \makeop{supp}  \makeop{ker}   \makeop{im}    \makeop{dom}
\makeop{coker} \makeop{Stab}  \makeop{SO}    \makeop{SL}    \makeop{SL}
\makeop{Cl}    \makeop{cond}  \makeop{Br}    \makeop{inv}   \makeop{rank}
\makeop{id}    \makeop{Fil}   \makeop{Frac}  \makeop{GL}    \makeop{SU}
\makeop{Nrd}   \makeop{Sp}    \makeop{Tr}    \makeop{Trd}   \makeop{diag}
\makeop{Res}   \makeop{ind}   \makeop{depth} \makeop{Tr}    \makeop{st}
\makeop{Ad}    \makeop{Int}   \makeop{tr}    \makeop{Sym}   \makeop{can}
\makeop{length}\makeop{SO}    \makeop{torsion} \makeop{GSp} \makeop{Ker}
\makeop{Adm}   \makeop{Mat}   \makeop{Cor}   \makeop{Inf} \makeop{Cok}
\makeop{Ver}
\def\makebb#1{\expandafter\def
  \csname bb#1\endcsname{{\mathbb{#1}}}\ignorespaces}
\def\makebf#1{\expandafter\def\csname bf#1\endcsname{{\bf
      #1}}\ignorespaces} 
\def\makegr#1{\expandafter\def
  \csname gr#1\endcsname{{\mathfrak{#1}}}\ignorespaces}
\def\makescr#1{\expandafter\def
  \csname scr#1\endcsname{{\EuScript{#1}}}\ignorespaces}
\def\makecal#1{\expandafter\def\csname cal#1\endcsname{{\mathcal
      #1}}\ignorespaces} 
\def\doLetters#1{#1A #1B #1C #1D #1E #1F #1G #1H #1I #1J #1K #1L #1M
                 #1N #1O #1P #1Q #1R #1S #1T #1U #1V #1W #1X #1Y #1Z}
\def\doletters#1{#1a #1b #1c #1d #1e #1f #1g #1h #1i #1j #1k #1l #1m
                 #1n #1o #1p #1q #1r #1s #1t #1u #1v #1w #1x #1y #1z}
\doLetters\makebb   \doLetters\makecal  \doLetters\makebf
\doLetters\makescr 
\doletters\makebf   \doLetters\makegr   \doletters\makegr
     \def\qed{\qedmark\medbreak}%
\def\qedmark{{\enspace\vrule height 6pt width 5pt depth 1.5pt}}%
    \def\setminus{\smallsetminus}

\def\wh{\widehat}
\def\wt{\widetilde}
\def\Ind{{\rm Ind}}

\def\G{\mathbb{G}}

\def\Q{\mathbb{Q}}
\def\Z{\mathbb{Z}}
\def\A{\mathbb{A}}

\def\gal{\text{Gal}}

\def\embed{\hookrightarrow}
\def\F{\bbF}

\newcommand{\isoto}{\stackrel{\sim}{\to}}

\makeatletter
\newcommand{\xdashrightarrow}[2][]
  {\ext@arrow 0359\rightarrowfill@@{#1}{#2}}
\newcommand{\xdashleftarrow}[2][]
  {\ext@arrow 3095\leftarrowfill@@{#1}{#2}}
\newcommand{\xdashleftrightarrow}[2][]{\ext@arrow 3359\leftrightarrowfill@@{#1}{#2}}
\def\rightarrowfill@@{\arrowfill@@\relax\relbar\rightarrow}
\def\leftarrowfill@@{\arrowfill@@\leftarrow\relbar\relax}
\def\leftrightarrowfill@@{\arrowfill@@\leftarrow\relbar\rightarrow}
\def\arrowfill@@#1#2#3#4{%
  $\m@th\thickmuskip0mu\medmuskip\thickmuskip\thinmuskip\thickmuskip
   \relax#4#1
   \xleaders\hbox{$#4#2$}\hfill
   #3$%
}
\makeatother

\usepackage[OT2,T1]{fontenc}
\DeclareSymbolFont{cyrletters}{OT2}{wncyr}{m}{n}
\DeclareMathSymbol{\Sha}{\mathalpha}{cyrletters}{"58}

\def\Gmk{\G_{{\rm m}, k}}

\def\GmK{\G_{{\rm m}, K}}

\begin{document}

\title[Class numbers of multinorm-one tori]{Class numbers of multinorm-one tori}

\author{Fan-Yun Hung}
\address{(Hung) Department of Mathematics, UCLA, 
520 Portola Plaza, Los Angeles, CA 90095-1555, USA}
\email{hfy880916@gmail.com}

% \author{
% \name{Jia-Wei Guo}
% \address{(Guo) Department of Mathematics, National Taiwan
% University, No.~1, Roosevelt Rd. Sec.~4  
% Taipei, Taiwan, 10617}
% \email{jiaweiguo312@gmail.com}
% }

% \author{
% \name{Nai-Heng Sheu}
% \address{(Sheu) Department of Mathematics, Indiana University, Rawles Hall, 831 East 3rd St. Bloomington, IN, USA, 47405}
% \email{naihsheu@iu.edu}
% }

\author{Chia-Fu Yu}
\address{(Yu) Institute of Mathematics, Academia Sinica and NCTS,
6F Astronomy Mathematics Building, No.~1, Roosevelt Rd. Sec.~4  
Taipei, Taiwan, 106319}
\email{chiafu@math.sinica.edu.tw}

% \author{
% \name{Chia-Fu Yu}
% \address{(Yu) Institute of Mathematics, Academia Sinica and NCTS,
% 6F Astronomy Mathematics Building, No.~1, Roosevelt Rd. Sec.~4  
% Taipei, Taiwan, 10617}
% \email{chiafu@math.sinica.edu.tw}
% }

\date{\today}
 % {\it Date}:\ \today} 

\maketitle

\begin{abstract}
We present a formula for the class number of a multinorm one torus $T_{L/k}$ associated to any \'etale algebra $L$ over a global field $k$. This is deduced from a formula for analogues of invariants introduced by T.~Ono, which are interpreted as a generalization of Gauss genus theory. This paper includes the variants of Ono's invariant for arbitrary $S$-ideal class numbers and the narrow version, generalizing results of Katayama, Morishita, Sasaki and Ono.   
\end{abstract}

\section{Introduction}
\label{sec:I}

Let $K/k$ be a finite extension of number fields. Ono
\cite{ono:genus85,ono:nagoya1987} introduced the
following alternating products of class numbers
\begin{equation}
  \label{eq:I.1}
  E(K/k)\coloneqq \frac{h(K)}{h(k)\cdot h(R_{K/k}^{(1)} \GmK)} \quad
  \text{and} \quad  
  E^+(K/k)\coloneqq \frac{h^+(K)}{h^+(k)\cdot h^+(R_{K/k}^{(1)}
  \GmK )}, 
\end{equation}
where $R_{K/k}$ denotes the Weil restriction of scalars from $K$ to
$k$, $R_{K/k}^{(1)} \GmK \subset R_{K/k} \GmK $ is the norm one torus, and $h$ (resp.~$h^+$) denotes the class number (resp.~the narrow class number). 
When the extension $K/k$ is Galois, Ono computed these invariants in terms of cohomological invariants \cite[Section 2, Theorem, p.~123]{ono:nagoya1987} and 
thus gave a class number relation among $h(K)$, $h(k)$ and $h(R_{K/k}^{(1)}  \GmK)$, as well as their narrow variants. Particularly, this yields a formula for the class number $h(R_{K/k}^{(1)}  \GmK)$ of the norm one torus $R_{K/k}^{(1)}  \GmK$. Restricted to the special case where $K$ is a CM field and $k=K^+$ is its maximal totally real subfield, Ono's formula gives an alternative proof of the following class number formula (see \cite[(16), p.~375]{Shyr-class-number-relation})
\begin{equation}\label{eq:CM}
    h(R_{K/K^{+}}^{(1)}  \GmK)=\frac{h_{K}}{h_{K^+}} \frac{1}{Q_{K/K^+} \cdot 2^{t-1}}, 
\end{equation}
where $Q_{K/K^+}=[O_K^\times: \mu_K O_{K^+}^\times]$ is the Hasse unit index and $t$ is the number of finite places of $K^+$ ramified in $K$. 
This formula was applied to compute the number of polarized CM complex abelian varieties in \cite{guo-sheu-yu:CM}.

Ono observed that the class number relation deduced from $E^+(K/k)$ generalizes Gauss's theorem on the genera of quadratic forms. For example, in the simplest case where 
$K$ is any quadratic extension of $k=\Q$, Ono's formula for $E^+(K/k)$ reads \[ h^+(K)=h_K^* \cdot 2^{t-1}, \] where $t$ is the number of rational primes ramified in $K$ and $h_K^*$ is the class number of
any genus in the narrow ideal class group $\Cl^+(K)$. On the other hand, when $K/k$ is any cyclic Kummer extension, Ono's formula shows a direct relation with the ambiguous class number for $K/k$; see \cite[Equation (10)]{ono:genus85} for details. We refer for a few explorations of ambiguous class numbers to 
\cite[Chapter XIII, Section 4]{lang-CycloF}, \cite{gras:73} and \cite{li-yu:gras-chevalley}.

There are generalizations and extensions of Ono's work by several other authors. 
In~\cite{sasaki:nagoya88} Sasaki gave a more direct proof of Ono's formulas which avoids $K$-theory. The formulas were generalized by Katayama~\cite{katayama:nagoya1989,katayama:kyoto1991} for any finite extension $K/k$ using Ono-Shyr's formula \cite{Shyr-class-number-relation} for isogenous tori. 
Morishita \cite{morishita:nagoya1991} generalized Ono's formula to the $S$-arithmetical setting including the function field case (still, as Sasaki and Ono, assuming $K/k$ Galois). He adopted a new approach using Nisnevich cohomology and gave a different approach. As another generalization, Morishita also showed a formula for the Ono invariant associated to the product $K_1\times K_2$ of two linearly disjoint Galois extensions $K_1$ and $K_2$, relating to  H\"urlimann's result \cite{hurlimann} on the Hasse norm principle for $K_1\times K_2$. \\ 

%A crucial point for getting a simpler final formula is due to the results that the Tate-Shaferevich group in question vanishes, and that \\

In this paper we generalize the results of Ono, Sasaki, Katayama and Morishita to an arbitrary \'etale algebra over any global function $k$, including an arbitrary $S$-arithmetical setting (i.e. $S$ is nonempty or not). Our approach is close to that of Sasaki, which does not rely on  $K$-theory nor the Nisnevich cohomology and is more elementary.   

Let $L=\prod_{i=1}^r K_i$ be an \'etale algebra over a global field $k$ with finite separable field extensions $K_i/k$, and let
$N_{L/k}=\prod_{i=1}^r N_{K_i/k}$ be the norm map from $L$ to $k$. 
Set 
\[ T_{L/k}:=\ker (N_{L/k}: T^L:=R_{L/k} \G_{{\rm m}, L}\to \Gmk),\]
called the \emph{multinorm-one torus} $T_{L/k}$ associated to $L/k$.
%is the
%kernel of the homomorphism $N_{L/k}: T^L:=R_{L/k} \G_{{\rm m}, L}\to \Gmk$ of $k$-tori. 
For simplicity, we write $N$ for $N_{L/k}$.

Let $\A_k$ and $\A_L:=\prod_{i} \A_{K_i}$ be the adele rings of $k$ and $L$, respectively.
Let $S$ be a nonempty finite set of places of
$k$ which contains all archimedean places if $k$ is a number field.
Let $\A_{k,S}$ and $\A_{L,S}$ be the $S$-adele rings of $k$ and $L$, respectively; see \eqref{eq:Sadele}.   
Let $O_{k,S} \coloneqq k\cap \A_{k,S}$ and $O_{L,S} \coloneqq L\cap \A_{L,S}$ be the
$S$-rings of integers in $k$ and $L$, respectively. 
Denote by $U_{k,S} \coloneqq \A_{k,S}^\times$ (resp.~$U_{L,S} \coloneqq \A_{L,S}^\times$) the unit group of $\A_{k,S}$ (resp.~$\A_{L,S}$). 

For any $k$-torus $T$, let 
\[  \Cl_S(T) \coloneqq \frac{T(\A_k)}{T(k) U_{T,S}} \quad \text{ and }  h_S(T):=|\Cl_S(T)| \]  denote 
the \emph{$S$-class group} and \emph{$S$-class number} 
of $T$, respectively; see \eqref{eq:Cl_S(T)}. If $k$ is a number field, we let 
$\Cl^+_S(T)$ and $h^+_S(T)$ denote the \emph{narrow $S$-class group} and 
\emph{narrow $S$-class number} 
of $T$, respectively; see \eqref{eq:Cl^+_S(T)}. 
Following Ono, we define the following alternating products:
\begin{equation}
  \label{eq:I.2}
E_S(L/k) \coloneqq \frac{h_S(L)}{h_S(k) h_S(T_{L/k})} \quad \text{and} 
\quad  E^+_S(L/k)\coloneqq \frac{h^+_S(L)}{h_S^+(k) h^+_S(T_{L/k})},
\end{equation}
where $h_S^{(+)}(L):=h_S^{(+)}(T^L)$ and $h_S^{(+)}(k):=h_S^{(+)}(\Gmk)$ are the (narrow) $S$-class numbers of $L$ and $k$, respectively. 

In the case where $k$ is a global function field and $S=\emptyset$, the class group $\Cl_\emptyset (T)=:\Cl(T)$ of a $k$-torus $T$ may be infinite. So instead we consider the class group $\Cl^0(T)\subset \Cl(T)$ of degree zero of $T$ (see \eqref{eq:Cl0T}) and the class number $h^0(T):=|\Cl^0(T)|$ of degree zero. Set
\begin{equation}
  \label{eq:I.3}
E^0(L/k) \coloneqq \frac{h^0(L)}{h^0(k) h^0(T_{L/k})},   
\end{equation}
where $h^0(L) \coloneqq h^0(T^L)$, $h^0(k) \coloneqq \abs{\A_k^{\times, 0}/k^{\times} \cdot U_k}$ and $U_k=\prod_{v} O_{k_v}^\times$. Let $U_L:=\prod_{i=1}^r U_{K_i}^\times$ be the unit subgroup of $\A_L^\times$. 

\begin{comment}

We say the
\emph{multiple norm principle} holds if 
\begin{equation}
  \label{eq:I.2}
  k^\times \cap N_{L/k} (\A_L^\times) = N_{L/k}(L^\times),
\end{equation}
where $\A_L=\prod_{i=1}^r \A_{K_i}$ is the adele ring of $L$. The
\emph{Tate-Shafarevich group} of $L/k$
\begin{equation}
  \label{eq:I.3}
  \Sha(L/k)\coloneqq  \left(k^\times \cap N_{L/k} \A_L^\times\right) / 
N_{L/k}(L^\times)
\end{equation}
measures the failure of the multiple norm principle. It is well-known
that $\Sha(L/k)$ is canonically isomorphic to the Tate-Shafarevich
group $\Sha^1(k,T_{L/k})$ of $T_{L/k}$
The multiple norm principle and the Tate-Shafarevich groups
$\Sha(L/k)$ have been
investigated by H\"urlimann~\cite{hurlimann}, 
Prasad and Rapinchuk~\cite{prasad-rapinchuk}, Pollio and
Rapinchuk~\cite{pollio-rapinchuk, pollio},
Demarche and D. Wei~\cite{demarche-wei},
Bayer-Fluckiger, T.-Y. Lee and Parimala~\cite{BLP19,lee2022tate} and
others. 

In this paper we define the
Ono invariant 
\begin{equation}
  \label{eq:I.4}
  E(L/k)\coloneqq \frac{h(L)}{h(k)\cdot h(T_{L/k})}, 
\end{equation}

\end{comment}

To describe our main results, we set some more notation. 
%Let $\A_{k,S}$ and $\A_{L,S}$ be the $S$-adele rings of $k$ and $L$, respectively.   
%Let $O_{k,S} \coloneqq k\cap \A_{k,S}$ and $O_{L,S} \coloneqq L\cap \A_{L,S}$ be the
%$S$-rings of integers in $k$ and $L$, respectively. 
%Denote by $U_{k,S} \coloneqq \A_{k,S}^\times$ (resp.~$U_{L,S} \coloneqq \A_{L,S}^\times$) the unit group of $\A_{k,S}$ (resp.~$\A_{L,S}$). 
% \textcolor{red}{I think this should be ``Denote by $U_{k, S}$ (resp. $U_{L, S}$) the unit group of $\A_{k, S}^{\times}$ (resp. $\A_{L, S}^{\times}$).''}
Let 
\begin{equation}
    \Sha(L/k):=\frac{k^{\times} \cap N(\A_{L}^{\times})} {N(L^\times)}
\end{equation}
denote the Tate-Shafarevich group of $L/k$. For any map $\alpha:A\to B$ of abelian groups, the $q$-symbol of $\alpha$ is defined by 
\begin{equation}\label{eq:q}
    q(\alpha):=\frac{|\coker \alpha|}{|\ker \alpha|}
\end{equation}
%\[ q(\alpha):=|\coker \alpha|/|\ker \alpha|$ 
if both 
$\coker \alpha$ and $\ker \alpha$ are finite. If $k$ is a number field, 
we refer to \eqref{eq:TAk+} for the definition of subgroups $\A_k^{\times,  +}\subset \A_k^\times$ and $\A_L^{\times,+}\subset \A_L^{\times}$, and for any subgroup $A\subset \A_k^\times$ (resp.~$A\subset \A_L^\times$), set $A^+:=A\cap \A_{k}^{\times,+}$ (resp. $A^+:=A\cap \A_{L}^{\times,+}$).

Our main results give formulas for the invariants
$E_S(L/k)$, $E_S^+(L/k)$ and $E^0(L/k)$: 
% see Theorems~\ref{thm:ESLk} and \ref{thm:E0Lk}.

\begin{thm}[Theorems~\ref{thm:ESLk} and \ref{thm:E0Lk}]\label{Main:1}
  Let $L$ be an \'etale algebra over a global field $k$. 
\begin{enumerate}
    \item When $S$ is nonempty, we have
  \begin{equation}
    \label{eq:I.4}
   E_S(L/k) = \frac{\abs{\Sha(L/k)}}{[L_{ab}:k]}  \cdot \frac{[U_{k, S}: N(U_{L, S})]}{[ {O}_{k, S}^{\times}: N( {O}_{L, S}^{\times})]}
  \end{equation}
and 
  \begin{equation}
    \label{eq:I.5}
   E_S^+(L/k) = \frac{\abs{\Sha(L/k)}}{[L_{ab}:k] \cdot q(\phi)}  \cdot \frac{[U_{k, S}: N(U_{L, S})]}{[ {O}_{k, S}^{\times +}: N( {O}_{L, S}^{\times +})]},
  \end{equation}
   where $L_{ab}$ is the maximal abelian extension of $k$ that is
   contained in all $K_i$ and $\phi: k^{\times +}/N (L^{\times +}) \to k^{\times} / N (L^{\times})$ is the canonical homomorphism.  

   \item When $k$ is a global function field and $S$ is empty, we have 
   \begin{equation}
       \label{eq:I.6}
          E^0(L/k) = \frac{\abs{\Sha(L/k)}}{[L_{ab}: k]} \cdot q(\phi^0) \cdot \frac{[U_k : N(U_L)]}{[\mathbb{F}_q^{\times} : N(\prod_i \mathbb{F}_{q_i}^{\times})]},         
   \end{equation}
   where $\phi^0: \A_k^{\times, 0}/N(\A_L^{\times, 0}) \to \A_k^{\times}/N(\A_L^{\times})$ is the canonical homomorphism.  
  \end{enumerate}  
\end{thm}

When $k$ is a number field and $S=\infty$ is the set of archimedean places, the class number $h_S(T)$ will be denoted by $h(T)$.  

\begin{cor} \label{cor:class_number_form}
    Let $L=\prod_{i=1} K_i$ be an \'etale algebra over a number field and $T_{L/k}$ be the associated 
    multinorm-one $k$-torus. Then 
\begin{equation}\label{eq:I.10}
    h(T_{L/k})=\frac{h(L)}{h(k)}\cdot \frac{[L_{ab}:k]}{\abs{\Sha(L/k)}} \cdot \frac{[ {O}_{k}^{\times}: N( {O}_{L}^{\times})]} {[U_{k}: N(U_{L})]}.
\end{equation}        
\end{cor}

Using Ono's formula on Tamagawa numbers of tori \cite[Main Theorem, p.~68]{ono:tamagawa} (also see \cite[Chapter IV, Corollary 3.3, p.~56]{oesterle}), one observes that $[L_{ab}:k]/\abs{\Sha(L/k)}$ is equal to the Tamagawa number $\tau(T_{L/k})$ of $T_{L/k}$. The formula~\eqref{eq:I.10} can also be written as
\begin{equation}\label{eq:I.11}
    h(T_{L/k})=\frac{h(L)}{h(k)}\cdot \tau(T_{L/k}) \cdot \frac{[ {O}_{k}^{\times}: N( {O}_{L}^{\times})]} {[U_{k}: N(U_{L})]}.
\end{equation}   \

%We also include the case of global function fields
%and give formulas for the invariant 
%$E_S(L/k)={h_S(L)}/\left ( h_S(k)\cdot h_S(T_{L/k}) \right )$ 
%of $S$-class numbers and its narrow version $E_S^+(L/k)$. The case where $k$ is global function field and $S$ is empty needs to be restricted to the class numbers of ``divisors of degree zero'' (otherwise the class group may be infinite) and is
%treated separately.

We give the proof of Theorem~\ref{Main:1} in Sections~\ref{sec:S} and \ref{sec:Z}. In Section~\ref{sec:E}, we explore the terms of our formulas in Theorem~\ref{Main:1} and give a few examples. Some of them are classical computational problems, for example, computing the unit group $O_k^\times$ of $O_k$. We also discuss some very recent results on the group $\Sha(L/k)$ and indicate particularly that the term $|\Sha_k(T')|$ in the main theorem of Morishita~\cite[p.~135]{morishita:nagoya1991} is always equal to $1$.

\section{$S$-class numbers of multinorm-one tori}
\label{sec:S}
Let $k$ be a global field and $k_s$ a fixed separable closure of $k$. 
Let $L=\prod_{i=1}^r K_i$ be an \'etale $k$-algebra, 
where each $K_i$ is
a separable field extension of $k$ in $k_s$. Let $\A_k$ and
$\A_L=\prod_{i=1}^r \A_{K_i}$ be the adele rings of $k$ and $L$,
respectively. For each place $v$ of $k$, 
denote by $k_v$ the completion of
$k$ at $v$ and set $L_v := L\otimes_k k_v =\prod_{w|v} L_w$. 
Here a place $w$ of $L$ is a place of
$K_i$ for some $i$ and its completion $L_{w}$ is simply $K_{i,w}$.
If $v$ is finite, let $O_v$ be
the valuation ring of $k_v$ and $O_{L_v}$ the maximal $O_v$-order 
of $L_v$, which is the product of the valuation rings $O_{L_w}$ of
$L_w$ for all $w|v$.
 
Let $S$ be a nonempty finite set of places of
$k$ which contains all archimedean places if $k$ is a number field. 
The $S$-adele rings of $k$ and $L$ are 
\begin{align}\label{eq:Sadele}
    \A_{k,S} \coloneqq \prod_{v\in S} k_v \times \prod_{v\not\in S} O_v, \quad \text{and} \quad 
    \A_{L,S} \coloneqq \prod_{v\in S} L_v \times \prod_{v\not\in S} O_{L_v}.
\end{align}
Let $O_{k,S} \coloneqq k\cap \A_{k,S}$ and $O_{L,S} \coloneqq L\cap \A_{L,S}$ be the
$S$-rings of integers in $k$ and $L$, respectively. 
Denote by $U_{k,S} \coloneqq \A_{k,S}^\times$ (resp.~$U_{L,S} \coloneqq \A_{L,S}^\times$) the unit group of $\A_{k,S}$ (resp.~$\A_{L,S}$). 
% \textcolor{red}{I think this should be ``Denote by $U_{k, S}$ (resp. $U_{L, S}$) the unit group of $\A_{k, S}^{\times}$ (resp. $\A_{L, S}^{\times}$).''}
Let $N_{L/k}: \A_L^\times \to \A_k^\times$ be the norm map,
and let $\A_L^{(1)}\coloneqq \ker N_{L/k} \subset \A_L^\times $ 
be the norm one subgroup. For
any subgroup $A\subset \A_L^\times$, denote by $A^{(1)}\coloneqq A\cap
\A_L^{(1)}$ its norm one subgroup.

%If $k$ is a number field,
%denote by $A^+$ the subgroup consisting of elements $(a_w)_w$ such
%that $a_w>0$ for any real place $w$.
%And for any subgroup $G \subset k^{\times}$, we denote by $G^+$ its subgroup consisting of totally positive elements.

The \emph{$S$-class group} and \emph{$S$-class number} 
of a $k$-torus $T$ are defined as   
\begin{equation}
  \label{eq:Cl_S(T)}
  \Cl_S(T) \coloneqq \frac{T(\A_k)}{T(k) U_{T,S}}, \quad h_S(T) \coloneqq \abs{\Cl_S(T)}, 
\end{equation}
where $U_{T,S}\coloneqq T(\A_{k,S})=\prod_{v\in S} T(k_v) \times
\prod_{v\not\in S} T(O_v)$ is the $S$-unit subgroup of $T(\A_k)$. 

When $k$ is a number field, we let 
\begin{equation}\label{eq:TAk+}
T(\A_k)^+ \subset T(\A_k)    
\end{equation}
denote the subgroup consisting of elements $(x_v)$, such that $x_v$ lies in the neutral component $T(k_v)^0$ of the Lie group $T(k_v)$ for all real places $v$. 
For any subgroup $A \subset T(\A_k)$, define $A^+ \coloneqq A \cap T(\A_k)^+$.
The \emph{narrow $S$-class group} and 
\emph{narrow $S$-class number} 
of a $k$-torus $T$ are defined as   
\begin{equation}
  \label{eq:Cl^+_S(T)}
  \Cl^+_S(T)\coloneqq \frac{T(\A_k)}{T(k) U^+_{T,S}}, \quad h^+_S(T) \coloneqq \abs{\Cl^+_S(T)}.
\end{equation}

Denote by $\A_k^S\coloneqq \prod_{v\notin S}' k_v$ the prime-to-$S$ adele ring of $k$.

\begin{lemma}\label{lm:hS+}
  If $k$ is a number field, we have $\Cl_S^+(T)\simeq T(\A_k)/T(k)^+  U_{T,S}$.
\end{lemma}
\begin{proof}
Clearly, $T(\A_k)/T(k)^+ U_{T,S}=T(\A_k^S)/T(k)^+ U_T^S$, where $U_T^S=\prod_{v\not\in S} T(O_v)$ is the maximal open subgroup of $T(\A_k^S)$.
Since $T$ is connected, real approximation implies that 
$T(k)\subset T(k_\infty)=\prod_{v|\infty} T(k_v)$ is dense. Thus, $T(\A_k)=T(k) T(\A_k)^+$ and the surjective map 
$T(\A_k)^+ \to T(k) T(\A_k)^+/T(k) U_{T,S}^+$ induces an isomorphism $T(\A_k)^+/T(k)^+ U_{T,S}^+ \isoto T(\A_k)/T(k) U_{T,S}^+$, while the left hand side is $T(\A_k^S)/T(k)^+ U_T^S$. This proves the lemma. \qed
\end{proof}

Let $I_{L, S} = \prod_{i = 1}^r I_{K_i, S}$ (resp.~$P_{L, S} = \prod_{i = 1}^r P_{K_i, S}$, $P_{L, S}^+ = \prod_{i = 1}^r P_{K_i, S}^+$) be the group of fractional ideals (resp. principal ideals, principal ideals generated by totally positive elements) of $L$ which are prime to $S$.
We set \[ \Cl_S(L) \coloneqq \prod_{i = 1}^r \Cl_S(K_i) = \prod_{i = 1}^r I_{K_i, S}/P_{K_i, S} \] 
(resp.~$\Cl_S^+(L) \coloneqq \prod_{i = 1}^r \Cl_S^+(K_i) = \prod_{i = 1}^r I_{K_i, S}/P_{K_i, S}^+$) to be the $S$-ideal class group (resp.~narrow $S$-ideal class group) of $L$.
We denote by 
\[ \text{$h_S(L):=|\Cl_S(L)|\ $ and  $\ h_S^+(L):=|\Cl_S^+(L)|$} \] 
the $S$-class number and narrow $S$-class number of $L$, respectively.
Furthermore, we denote by $I_{L, S}^{(1)}$ the kernel of the norm map $N_{L/k}: I_{L, S} \to I_{k, S}$ and write $P_{L, S}^{(1)} = P_{L, S} \cap I_{L, S}^{(1)}$.

If $T=R_{L/k} \G_{{\rm m}, L}$, then we have
$\Cl_S(L) = \Cl_S(T) = \A_L^\times /L^\times U_{L,S}$ and $\Cl_S^+(L) = \Cl_S^+(T) = \A_L^\times /L^\times U^+_{L,S}$, the $S$-ideal class group of $L$ and
its narrow version,  
and we have $h_S(T)=h_S(L)$ and 
$h_S^+(T)=h_S^+(L)$. 
Recall that $T_{L/k}\coloneqq \ker \left(N_{L/k}:R_{L/k} \G_{{\rm m}, L}\to \Gmk \right)$ 
denotes the multinorm-one torus associated to $L/k$. We have
\begin{align*}
    h_S(T_{L/k})=[\A_L^{(1)}: L^{(1)} U_{L,S}^{(1)}] \quad \text{and}\quad h^+_S(T_{L/k})=[\A_L^{(1)}: L^{(1)} U_{L,S}^{(1)+}]=[\A_L^{(1)}: L^{(1)+} U_{L,S}^{(1)}],
\end{align*}
where $U_{L,S}^{(1)+}=U_{L,S}^{(1)}\cap U_{L,S}^{+}$ and $L^{(1)+}=L^{(1)}\cap \A_{L}^{(1)+}$.
Following Ono, we extend the definition of the invariants in \eqref{eq:I.1} to $L/k$:

\begin{defn}
    Let $L/k$ and $S$ be as above, the \emph{Ono invariant} and its narrow version are defined as
\begin{equation}
  \label{eq:ESLk}
  E_S(L/k) \coloneqq \frac{h_S(L)}{h_S(k) h_S(T_{L/k})} \quad \text{and} 
\quad  E^+_S(L/k)\coloneqq \frac{h^+_S(L)}{h_S^+(k) h^+_S(T_{L/k})},
\end{equation}
where the narrow version is defined only when $k$ is a number field. \end{defn}

\begin{prop} \label{prop: Sasaki 2}
We have
\begin{align}
    h_S(T_{L/k}) &= \frac{[I_{L, S}^{(1)}: P_{L, S}^{(1)}] [ {O}_{k, S}^{\times} \cap N(L^{\times}): N({O}_{L, S}^{\times})]}{[U_{k, S}
    \cap N(\A_{L}^{\times}): N(U_{L, S})]}; \label{eq:hSTLk}\\
    h_S^+(T_{L/k}) &= \frac{[I_{L, S}^{(1)}: P_{L, S}^{(1)+}] [ {O}_{k, S}^{\times +} \cap N(L^{\times +}): N( {O}_{L, S}^{\times +})]}{[U_{k, S}
    \cap N(\A_{L}^{\times}): N(U_{L, S})]}. \label{eq:hS+TLk}
\end{align}
\end{prop}

\begin{proof}
We prove \eqref{eq:hS+TLk}; the proof of \eqref{eq:hSTLk} is the same where one deletes ``+'' from the terms.  
Consider the following two commutative diagrams
\begin{center}
\begin{tikzcd} 
1 \arrow[r] & U_{L,S} \arrow[r] \arrow[d, "N"] & \A_{L}^\times \arrow[r] \arrow[d, "N"] &  I_{L,S} \arrow[r] \arrow[d, "N"] & 1 \\ 
1 \arrow[r] & U_{k, S} \arrow[r] & \A_{k}^{\times} \arrow[r] & I_{k,S} \arrow[r] & 1; 
\end{tikzcd}
\end{center}
\begin{center}
\begin{tikzcd} 
1 \arrow[r] & O_{L,S}^{\times +} \arrow[r] \arrow[d, "N"] & {L}^{\times +} \arrow[r] \arrow[d, "N"] &  P_{L,S}^+ \arrow[r] \arrow[d, "N"] & 1 \\ 
1 \arrow[r] & O_{k, S}^{\times +} \arrow[r] & {k}^{\times +} \arrow[r] & P_{k,S}^+ \arrow[r] & 1. 
\end{tikzcd}
\end{center}

The snake lemma gives the following two exact sequences
\begin{center}
\begin{tikzcd} 
1 \arrow[r] & U_{L,S}^{(1)} \arrow[r]  & \A_{L}^{(1)} \arrow[r]  &  I_{L,S}^{(1)}  \arrow[r, "\delta_1"]  & \dfrac{U_{k,S}}{N(U_{L,S})} \arrow[r] & \dfrac{\A_{k}^\times}{N(\A_{L}^\times)} \arrow[r] & \dotsb
\end{tikzcd}
\end{center}
\begin{center}
\begin{tikzcd}
1 \arrow[r] & O_{L,S}^{(1)+} \arrow[r]  & {L}^{(1)+} \arrow[r]  &  P_{L,S}^{(1)+} \arrow[r, "\delta_2"]  
& \dfrac{O_{k,S}^{\times +}} {N(O_{L,S}^{\times +})} 
\arrow[r] & \dfrac{k^{\times +}} {N({L}^{\times +})} \arrow[r] & \dotsb
\end{tikzcd}
\end{center}
Clearly, ${\rm Im}\, \delta_1=\left(U_{k,S} \cap N(\A_{L}^\times)\right) / N(U_{L,S})$ and ${\rm Im}\, \delta_2=\left(O_{k,S}^{\times +} \cap N({L}^{\times +})\right) /N(O_{L,S}^{\times +})$ and these two abelian groups are finite. Thus, we have a commutative diagram
\begin{center}
\begin{tikzcd} 
1 \arrow[r] 
& \A_{L}^{(1)}/U_{L,S}^{(1)}  \arrow[r]  
& I_{L,S}^{(1)}  \arrow[r] 
&  \left(U_{k,S} \cap N(\A_{L}^\times)\right) / N(U_{L,S}) \arrow[r]    
& 1 \\ 
1 \arrow[r] 
& {L}^{(1)+}/O_{L,S}^{(1)+} \arrow[r] \arrow[u, hook, "f'"]
& P_{L,S}^{(1)+} \arrow[r] \arrow[u, hook, "f"]
& \left(O_{k,S}^{\times +} \cap N({L}^{\times +})\right) /N(O_{L,S}^{\times +}) \arrow[r] \arrow[u, hook,  "f''"]
& 1. 
\end{tikzcd}
\end{center}
Therefore, 
\[ \begin{split}
    [  I_{L,S}^{(1)} :  P_{L,S}^{(1)+}]&=q(f)=q(f') q(f'') \\
    &=[\A_{L}^{(1)}: {L}^{(1)+} U_{L,S}^{(1)}] \frac{[U_{k,S} \cap N(\A_{L}^\times): N(U_{L,S})]}{[O_{k,S}^{\times +} \cap N({L}^{\times +}):N(O_{L,S}^{\times +})].}
\end{split}\]
This proves the proposition. \qed 
\end{proof}

For any $k$-torus $T$, we denote by $\widehat{T}$ its \emph{character group} $\Hom_{\overline{k}}(T, \Gmk)$. Also, for any finite commutative group $G$, let $G^\vee:=\Hom(G, \Q/\Z)$ denote the Pontryagin dual. 

\begin{thm} \label{thm: L_ab over k}
There are canonical isomorphisms
\begin{align}\label{eq:2.6}
    \frac{\A_k^{\times}}{k^{\times} \cdot N(\A_L^{\times})} \simeq H^1(\widehat{T}_{L/k})^{\vee} \simeq \Gal(L_{\rm ab}/k),
\end{align}
%\begin{center}
%    \begin{tikzcd}
%    \dfrac{\A_k^{\times}}{k^{\times} \cdot N(\A_L^{\times})} \arrow[r, "\sim"] & \Gal(L_{\rm ab}/k)
%    \end{tikzcd}
%\end{center}
where $L_{\rm ab}$ is the maximal abelian extension of $k$ that is contained in all $K_i$.
\end{thm}
\begin{proof}
Let $K$ be a finite Galois extension of $k$ containing all $K_i$ and $G = \gal(K/k)$. Denoting by $C_K$ the idele class group of $K$, one has the exact sequence
\begin{center}
    \begin{tikzcd}
    1 \arrow[r] & K^{\times} \arrow[r] & \A_K^{\times} \arrow[r] & C_K \arrow[r] & 1.
    \end{tikzcd}
\end{center}

Let $T$ be a $k$-torus splitting over $K$. Following Ono in \cite{ono:tamagawa}, we apply $\Hom(\widehat{T}, \cdot)$ to the above exact sequence with the canonical identifications,
\begin{align*}
    T(K) = \Hom(\widehat{T}, K^{\times}), \quad T(\A_K) = \Hom(\widehat{T}, \A_K^{\times})
\end{align*}
and define
\begin{align*}
    C_K(T) \coloneqq T(\A_K)/T(K) \simeq \Hom(\widehat{T}, C_K).
\end{align*}
The short exact sequence
\begin{align*}
    1 \longrightarrow T(K) \longrightarrow T(\A_K) \longrightarrow C_K(T) \longrightarrow 1
\end{align*}
induces the long exact sequence
\begin{align*}
    1 \to T(k) \to T(\A_k) \to C_K(T)^G \to H^1(T(K)) \to H^1(T(\A_K)) \to H^1(C_K(T)) \to \dotsb
\end{align*}
We claim that for $T = R_{L/k} \mathbb{G}_{m, L}$ or $T = \mathbb{G}_{m, k}$, we have $C_K(T)^G = T(\A_K)/T(k)$. If $T = \mathbb{G}_{m, k}$, this assertion holds since $H^1(G, T(K)) = H^1(k, \mathbb{G}_{m, k}) = 0$ by Hilbert's Theorem 90. If $T = R_{L/k} \mathbb{G}_{m, L}$, for each $K_i$, by Shapiro's lemma
we have 
\begin{align*}
    H^1(k, R_{K_i/k} \mathbb{G}_{m, K_i}) = H^1(K_i, \mathbb{G}_{m, K_i}) = 0
\end{align*}
and hence $H^1(G, T(K)) = H^1\left(k, \prod_{i = 1}^r R_{K_i/K} \mathbb{G}_{m, K_i}\right) = \prod_{i = 1}^r H^1(k, R_{K_i/K} \mathbb{G}_{m, K_i}) = 0$.

Now suppose we have the following exact sequence of $k$-tori splitting over $K$:
\begin{equation}\label{eq:2.9}
    \begin{tikzcd}
    1 \arrow[r] & T' \arrow[r, "\iota"] & T \arrow[r, "N"] & T'' \arrow[r] & 1.
    \end{tikzcd}
\end{equation}

Putting in $K \xhookrightarrow{} \A_K$, we obtain the commutative diagram
\begin{center}
    \begin{tikzcd}
    1 \arrow[r] & T'(K) \arrow[r, "\iota"] \arrow[d, hook] & T(K) \arrow[r, "N"] \arrow[d, hook] & T''(K) \arrow[r] \arrow[d, hook] & 1 \\
    1 \arrow[r] & T'(\A_K) \arrow[r, "\iota"] & T(\A_K) \arrow[r, "N"] & T''(\A_K) \arrow[r] & 1.
    \end{tikzcd}
\end{center}
By the snake lemma, we have the short exact sequence
\begin{center}
    \begin{tikzcd}
    1 \arrow[r] & C_K(T') \arrow[r, "\iota"] & C_K(T) \arrow[r, "N"] & C_K(T'') \arrow[r] & 1,
    \end{tikzcd}
\end{center}
which induces the long exact sequence
\begin{center}
    \begin{tikzcd}
    \dotsb \arrow[r] & C_K(T)^G \arrow[r, "N"] & C_K(T'')^G \arrow[r] & H^1(G, C_K(T')) \arrow[r, "\iota"] & H^1(G, C_K(T)) \arrow[r] & \dotsb
    \end{tikzcd}
\end{center}
and the long exact sequence
\begin{equation}\label{eq:2.10}
   \begin{tikzcd}
    1 \arrow[r] & \coker N \arrow[r] & H^1(G, C_K(T')) \arrow[r, "\iota"] & H^1(G, C_K(T)) \arrow[r] & \dotsb.
    \end{tikzcd} 
\end{equation}
Taking \eqref{eq:2.9} to be 
\begin{equation}
    \label{eq:2.11}
      \begin{tikzcd}
    1 \arrow[r] &  T_{L/k}  \arrow[r, "\iota"] & R_{L/k} \mathbb{G}_{m, L}  \arrow[r, "N"] & \mathbb{G}_{m, k} \arrow[r] & 1,
    \end{tikzcd}
\end{equation}
%$1 \to T_{L/k} \xrightarrow[]{\iota} R_{L/k} \mathbb{G}_{m, L} \xrightarrow[]{N} \mathbb{G}_{m, k} \to 1$, 
we find that
\begin{align*}
    % \ker \iota \simeq 
    \coker N = \frac{\A_k^{\times}}{k^{\times} \cdot N(\A_L^{\times})}.
\end{align*}
We have 
\begin{align*}
    H^1(G, C_K(R_{L/k} \mathbb{G}_{m, L}))
    = \prod\limits_{i = 1}^r H^1(\gal(K/k), C_K(R_{K_i/k} \mathbb{G}_{m, K_i}))
\end{align*}
where for each $i$,
\begin{align*}
    C_K(R_{K_i/k} \mathbb{G}_{m, K_i}) 
    = \frac{(K_i \otimes \A_K)^{\times}}{(K_i \otimes K)^{\times}} 
    = (\A_K^{\times}/K^{\times})^{[K:K_i]}.
\end{align*}
By Theorem 9.1 in \cite{tate1967global}, we have $H^1(\gal{(K/k)}, \A_K^{\times}/K^{\times}) = 0$, so $H^1(G, C_K(R_{L/k} \mathbb{G}_{m, L})) = 0$. Thus,
\begin{align*}
    \coker N \isoto H^1(G, C_K(T_{L/k})) \simeq H^1(G, \widehat{T}_{L/k})^{\vee},
\end{align*}
where the second isomorphism is Nakayama's duality \cite[Theorem 6.3]{platonov-rapinchuk}.

%From the exact sequence of $k$-tori
%\begin{center}
%    \begin{tikzcd}
%    1 \arrow[r] & T_{L/k} \arrow[r, "\iota"] & R_{L/k} %\mathbb{G}_{m, L} \arrow[r, "N"] & \mathbb{G}_{m, k} \arrow[r] & 1,
%    \end{tikzcd}
%\end{center}
Setting $T^L = R_{L/k} \mathbb{G}_{m, L}$ and taking duals from the exact sequence of $k$-tori \eqref{eq:2.11} , we obtain
\begin{center}
    \begin{tikzcd}
    1 \arrow[r] & \Z \arrow[r, "\widehat{N}"] & \widehat{T^L} \arrow[r, "\widehat{\iota}"] & \widehat{T}_{L/k} \arrow[r] & 1.
    \end{tikzcd}
\end{center}
Denoting the Galois group $\gal(K_i/k)$ by $H_i$, we obtain
\begin{align*}
    H^1(G, \widehat{T^L})
    = \bigoplus_{i = 1}^r H^1\left(G, \Ind_{H_i}^G \Z\right)
    = \bigoplus_{i = 1}^r H^1(H_i, \Z) = 0.
\end{align*}
We have the following long exact sequence
\begin{center}
    \begin{tikzcd}
    1 \arrow[r] & H^1(G, \widehat{T}_{L/k}) \arrow[r] & H^2(G, \Z) \arrow[r, "\widehat{N}"] \arrow[d, "\wr"] & \bigoplus\limits_{i = 1}^r H^2(H_i, \Z) \arrow[r] \arrow[d, "\wr"] & \dotsb \\
    & & \Hom(G, \Q/\Z) \arrow[r, "r"] & \bigoplus\limits_{i = 1}^r \Hom(H_i, \Q/\Z)
    \end{tikzcd}
\end{center}
where the map $r$ is  $f \mapsto (f|_{H_i})_i$. We describe $\ker r$ explicitly:
\begin{align*}
    H^1(G, \widehat{T}_{L/k}) = \ker r 
    &= \{f: G \to \Q/\Z \mid f|_{H_i} = 0 \text{ for all $i$ }\} \\
    &= \{f: G \to \Q/\Z \mid f|_{[G, G] H_1 \dotsb H_r} = 0 \} \\
    &\simeq \{f: G/([G, G]H_1 \dotsb H_r) \to \Q/\Z\}.
\end{align*}
The fixed field of $[G, G]H_1 \dotsb H_r$ is exactly $L_{\rm ab}$, the maximal abelian extension of $k$ contained in all $K_i$. Thus
\begin{align*}
    H^1(G, \widehat{T}_{L/k}) \simeq \Hom(\Gal(L_{\rm ab}/k), \Q/\Z) = \gal(L_{\rm ab}, k)^{\vee}.
\end{align*}
Altogether,
\begin{align*}
    \frac{\A_k^{\times}}{k^{\times} \cdot N(\A_L^{\times})} \simeq H^1(\widehat{T}_{L/k})^{\vee} \simeq \Gal(L_{\rm ab}/k),
\end{align*}
where each isomorphism is canonical. \qed
\end{proof}

\begin{rem}
    We can show the consequence $\A_k^\times/k^\times \cdot N(\A_L^\times)\simeq \Gal(L_{\rm ab}/k)$ of Theorem~\ref{thm: L_ab over k} directly. Indeed, by class field theory, the maximal abelian subextension $K_{i, \rm ab}$ of $K_i/k$ corresponds to the subgroup $k^\times \cdot N_{K_i/k}(\A_{K_i}^\times)$ of finite index. Thus, their intersection $\cap K_{i,\rm ab}=L_{\rm ab}$ corresponds to the subgroup $k^\times \cdot N(\A_{L}^\times)$ of finite index.
\end{rem}

\begin{thm}\label{thm:ESLk} 
Let the notation be as above.
We have 
  \begin{equation}
    \label{eq:ESLkformula}
   E_S(L/k) 
   = \frac{\abs{\Sha(L/k)}}{[L_{ab}:k]} 
   \cdot \frac{[U_{k, S}: N(U_{L, S})]}{[ {O}_{k, S}^{\times}: N( {O}_{L, S}^{\times})]}
  \end{equation}
and 
  \begin{equation}
    \label{eq:ES+Lkformula}
   E_S^+(L/k) 
   = \frac{\abs{\Sha(L/k)}}{[L_{ab}:k] \cdot q(\phi)} 
   \cdot \frac{[U_{k, S}: N(U_{L, S})]}{[ {O}_{k, S}^{\times +}: N( {O}_{L, S}^{\times +})]},
  \end{equation}
where $L_{ab}$ is the maximal abelian extension of $k$ contained in all $K_i$ and $\phi: k^{\times +}/N (L^{\times +}) \to k^{\times} / N (L^{\times})$ is the canonical homomorphism.
\end{thm}
\begin{proof}
We shall prove the formula \eqref{eq:ES+Lkformula}. The proof of \eqref{eq:ESLkformula} is the same where one deletes ``+'' from the terms. Let $\Tilde{a}: k^{\times}/N(L^{\times}) \rightarrow
\A_{k}^{\times}/N(\A_{L}^{\times})$ be the canonical
homomorphism and put $a \coloneqq  \Tilde{a} \circ \phi$. Consider the commutative diagram whose rows are exact: 
\begin{center}
\begin{tikzcd} 
1 \arrow[r] &  {O}_{k, S}^{\times +}/ {O}_{k, S}^{\times +} \cap N(L^{\times +}_S) \arrow[r] \arrow[d, "a'"] & k^{\times +}/N(L^{\times +}) \arrow[r] \arrow[d, "a"] & k^{\times +}/ {O}_{k, S}^{\times +} N(L^{\times +}) \arrow[r] \arrow[d, "a''"] & 1 \\ 
1 \arrow[r] & U_{k, S}/U_{k, S} \cap N(\A_{L}^{\times}) \arrow[r] & \A_{k}^{\times}/N(\A_{L}^{\times}) \arrow[r] & \A_{k}^{\times}/U_{k, S} N(\A_{L}^{\times}) \arrow[r] & 1. 
\end{tikzcd}
\end{center}

From the commutative diagram
\begin{center}
    \begin{tikzcd}
    1 \arrow[r] &  {O}_{k, S}^{\times +} \cap N(L^{\times +})/N( {O}_{L, S}^{\times +}) \arrow[r] \arrow[d, "b'"] &  {O}_{k, S}^{\times +}/N( {O}_{L, S}^{\times +}) \arrow[r] \arrow[d, "b"] &  {O}_{k, S}^{\times +}/ {O}_{k, S}^{\times +} \cap N(L^{\times +}) \arrow[r] \arrow[d, "a'"] & 1 \\ 
    1 \arrow[r] & U_{k, S} \cap N(\A_{L}^{\times})/N(U_{L, S}) \arrow[r] & U_{k, S}/N(U_{L, S}) \arrow[r] & U_{k, S}/U_{k, S} \cap N(\A_{L}^{\times}) \arrow[r] & 1 
    \end{tikzcd}
\end{center}
we have
\begin{align*}
    q(a') 
    = \frac{q(b)}{q(b')} 
    = \frac{[U_{k, S}: N(U_{L, S})]}{[ {O}_{k, S}^{\times +}:
    N( {O}_{L, S}^{\times +})]} 
    \cdot
    \frac{[ {O}_{k, S}^{\times +} \cap N(L^{\times +}):
    N( {O}_{L, S}^{\times +})]}{[U_{k, S} \cap N(\A_{L}^{\times}):N(U_{L, S})]}.
\end{align*}
Now, the diagram
\begin{center}
    \begin{tikzcd}
    1 \arrow[r] & P_{L, S}^+ \arrow[r] \arrow[d, "N'"] & I_{L, S} \arrow[r] \arrow[d, "N"] & \Cl_{L, S}^+ \arrow[r] \arrow[d, "N''"] & 1 \\ 
    1 \arrow[r] & P_{k, S}^+ \arrow[r] & I_{k, S} \arrow[r] & \Cl_{k, S}^+ \arrow[r] & 1
    \end{tikzcd}
\end{center}
induces the exact sequence 
% \textcolor{red}{Why the first arrow is surjective?}
\begin{center}
    \begin{tikzcd}
    I_{L, S}^{(1)} \arrow[r] & \ker N'' \arrow[r, "\delta"] & P_{k, S}^+/N'(P_{L, S}^+) \arrow[r, "c"] & I_{k, S}/N(I_{L, S}) \arrow[r, twoheadrightarrow] & \coker N'' \arrow[r] & 1 \\ 
    & & k^{\times +}/ {O}_{k, S}^{\times +} N(L^{\times +}) \arrow[r, "a''"] \arrow[u, "\wr"] & \A_{k}^{\times}/U_{L, S} N(\A_{L}^{\times}). \arrow[u, "\wr"] & 
    \end{tikzcd}
\end{center}
Then
\begin{align*}
\begin{split}
    q(a'') 
    & = q(c) 
    = \frac{\abs{\coker N''}}{\abs{\ker c}}
    = \abs{\coker N''} \cdot \frac{[I_{L, S}^{(1)}: P_{L, S}^{(1)+}]}{\abs{\ker N''}} \\
    & = q(N'') [I_{L, S}^{(1)}: P_{L, S}^{(1)+}]
    = [I_{L, S}^{(1)}: P_{L, S}^{(1)+}] \cdot \frac{h_S^+(k)}{h_S^+(L)}.    
\end{split}
\end{align*}
Thus using Proposition \ref{prop: Sasaki 2}, we have
\begin{align*}
    q(a) 
    = q(a') q(a'') 
    = \frac{[U_{k, S}: N(U_{L, S})]}{[ {O}_{k, S}^{\times +}: N( {O}_{L, S}^{\times +})]} \cdot E_S^+(L/k)^{-1}.
\end{align*}

On the other hand, 
% by Proposition \ref{prop: H1(T)},
\begin{align*}
    q(a) = q(\Tilde{a}) q(\phi)
    &= \frac{[\A_k^{\times}: k^\times N(\A_{L}^{\times})]}{[k^{\times} \cap N(\A_{L}^{\times}): N(L^{\times})]} \cdot q(\phi).
    % = \frac{[L_{ab}: k]}{[k^{\times} \cap N_{L/k}(\A_L^{\times}): N_{L/k}(L^{\times})]} \\
    % &= \frac{\#H^1(k, \widehat{T}_{L/k})}{\# \Sha^1(k, T_{L/k})}
    % = \tau(T_{L/k}).
\end{align*}
Using Theorem \ref{thm: L_ab over k}, we obtain the formula
\begin{align*}
    E_S^+(L/k) 
    = \frac{\abs{\Sha(L/k)}}{[L_{ab} : k] \cdot q(\phi)} \cdot \frac{[U_{k, S}: N(U_{L, S})]}{[ {O}_{k, S}^{\times +}: N( {O}_{L, S}^{\times +})]}. \ \text{\qed}
\end{align*}
\end{proof}

\section{Class numbers of degree zero}\label{sec:Z}

Let $k$ be a global function field and $T$ a $k$-torus. Here we fix a separable closure $k_s$ over $k$ and denote by $\Gamma_k$ the Galois group $\Gal(k_s/k)$. Considering the degree map $\deg_k: \A_k^{\times} \to \Z$, we define $\A_k^{\times, 0} \coloneqq \ker(\deg_k)$. Suppose that $\chi_1, \dotsc, \chi_r$ is a $\Z$-basis for $\widehat{T}^{\Gamma_k}$. Consider the map
\begin{align*}
    T(\A_k) \xrightarrow[]{\chi = (\chi_1, \dotsc, \chi_r)} (\A_k^{\times})^r \xrightarrow[]{(\deg_k)_i} \Z^r \xrightarrow{} 0.
\end{align*}
We define $T(\A_k)^0 \coloneqq \ker(\deg \circ \chi)$ and set the 
\emph{class group of degree zero of $T$} as
\begin{equation}\label{eq:Cl0T}
    \Cl^0(T) \coloneqq \frac{T(\A_k)^0}{T(k) U_T},
\end{equation}
where $U_T = \prod_v T(\mathcal{O}_v)$ is the maximal open compact subgroup of $T(\A_k)$. The order of this class group is finite an d is denoted by $h^0(T) \coloneqq \abs{\Cl^0(T)}$, called the \emph{class number of degree zero} of $T$.

Now let $L = \prod_{i = 1}^r K_i$ be an \'etale $k$-algebra and $T^L = R_{L/k} \mathbb{G}_{m,L}$. We set
\begin{equation}
    \label{eq:h0L}
    h^0(L) \coloneqq h^0(T^L), \quad h^0(k) \coloneqq \abs{\A_k^{\times, 0}/k^{\times} \cdot U_k}.
\end{equation}
%\begin{align*}
%    h^0(L) \coloneqq h^0(T^L), \quad h^0(k) \coloneqq \abs{\A_k^{\times, 0}/k^{\times} \cdot U_k}
%\end{align*}
Recall that $T_{L/k} = \ker(N_{L/k}: T^L \to \Gmk)$ is the multinorm-one torus associated to $L/k$. 
\begin{defn}
The \emph{Ono invariant of degree zero} is defined as
\begin{equation}\label{eq:E0Lk}
    E^0(L/k) \coloneqq \frac{h^0(T^L)}{h^0(k) h^0(T_{L/k})}.
\end{equation}   
\end{defn}
%We define the \emph{Ono invariant of degree zero}
%\begin{equation}\label{eq:E0Lk}
%    E^0(L/k) \coloneqq \frac{h^0(T^L)}{h^0(k) h^0(T_{L/k})}.
% \end{equation}

\begin{lemma}
If $k$ is a global function field, we have 
\begin{align*}
    h^0(T_{L/k}) = [\A_L^{(1)} \cap \A_L^{\times, 0}: L^{(1)} U_L^{(1)}],
\end{align*}
where $\A_L^{\times, 0} \coloneqq \prod_{i = 1}^r \A_{K_i}^{\times, 0}$.
\end{lemma}
\begin{proof}
Denoting the norm map by $N_i = N_{K_i/k}$, we have the following commutative diagram
\begin{center}
    \begin{tikzcd}
    & \A_L^{\times, 0} \arrow[d, hook] & \\
    T_{L/k}(\A_k) \arrow[r, hook] \arrow[rd, "\chi"] & \A_L^{\times}  \arrow[d, "(N_i)_i"] \arrow[rdd, bend left = 30, "N_{L/k}"] & \\
    & (\A_k^{\times})^r \arrow[d, twoheadrightarrow, "(\deg_k)_i"'] \arrow[rd, "\prod"] & \\
    & \Z^r & \A_k^{\times}
    \end{tikzcd}
\end{center}
where $\prod(a_1, \dotsc, a_r) = \prod_{i = 1}^r a_i$. We have
\begin{align*}
    T_{L/k}(\A_k)^0 
    = \left\{ x = (x_i) \in \A_L^{\times}: \prod\limits_{i = 1}^r N_i(x_i) = 1, \deg \circ N_i(x_i) = 0 \, \text{for all $i$}\, \right\} 
    = \A_L^{(1)} \cap \A_L^{\times, 0}.
\end{align*}
Also, since $L^{\times}$ and $U_L$ are already subsets of $\A_L^{\times, 0}$, we have $T_{L/k}(k) = L^{(1)} \cap \A_L^{\times, 0} = L^{(1)}$ and $U_{T_{L/k}} = U_L^{(1)} \cap \A_L^{\times, 0} = U_L^{(1)}$. This completes the proof. So \[ \Cl^0(T_{L/k})=T_{L/k}(\A_k)^0/T_{L/k}(k) U_{T_{L/k}}=(\A_L^{(1)} \cap \A_L^{\times, 0})/L^{(1)}U_L^{(1)}. \] 
This completes the proof. \qed
\end{proof}

In the following, for an abelian subgroup $A \subset \A_L$, we set $A^0 \coloneqq A \cap \A_L^{\times, 0}$.

Let $\F_q$ be the constant subfield of $k$ and $X$ the geometrically connected projective smooth algebraic curve over $\mathbb{F}_q$ with $k = \mathbb{F}_q(X)$. Denote by $\abs{X}$ the set of closed points of $X$. Recall that the group of divisors on $X$ is
\begin{align*}
    \Div(X) = \sum\limits_{P \in \abs{X}} n_P P, \quad n_P = 0 \text{ for all but finitely many } P \in \abs{X}
\end{align*}
and the degree map on $\Div(X)$,
\begin{align*}
    \deg(D) = \sum\limits_{P \in \abs{X}} n_P \deg(P) \quad \text{ for } D = \sum\limits_{P \in \abs{X}} n_P P, \text{ where } \deg(P) = [k(P): k].
\end{align*}
The group of divisors of degree zero is denoted as $\Div^0(X)$. Now for any rational function $f\in k^\times$ on $X$, we define
\begin{align*}
    \mathrm{div}(f) \coloneqq \sum\limits_{P \in \abs{X}} \ord_P(f) P
\end{align*}
and the group of principal divisors
\begin{align*}
    P(X) := \{ \mathrm{div}(f): f \in k^{\times}\}.
\end{align*}
Since $P(X) \subset \Div^0(X)$, we can consider the quotient $\Pic^0(X) \coloneqq \Div^0(X)/P(X)$. Finally we have the isomorphisms
\begin{align*}
    \Pic^0(X) \simeq \frac{\A_k^{(1)}}{k^{\times} \prod_v O_v^{\times}}, \quad \Div^0(X) \simeq \frac{\A_k^{(1)}}{\prod_v O_v^{\times}}.
\end{align*}

Recall $L=\prod_{i=1}^r K_i$. For each $i$, let $Y_i \to X$ be a finite cover over $\mathbb{F}_q$ with $K_i = \mathbb{F}_q(Y_i)$. Note that $Y_i$ may not be geometrically connected over $\mathbb{F}_q$.
%as $\mathbb{F}_q$ may not coincide with the constant subfield $\F_{q_i}$ of $K_i$. 
Let $\mathbb{F}_{q_i} = \Gamma(Y_i, \mathcal{O}_{Y_i})$ be the constant subfield of $K_i$. 
We define similarly $\Div(Y_i)$ to be the group consisting of $\sum\limits_{P \in \abs{Y_i}} n_P P$. Although different definitions of the degree of a divisor on $Y_i$ may differ by a constant, the notions of $\Div^0(Y_i)$ and $\Pic^0(Y_i)$ are still well-defined.

Similarly to the case when $k$ is a number field, we shall write $I_k^0 = \Div^0(X)$ and $P_k = P(X)$. We also define
\[ \text{$I_L^0 := \prod_{i = 1}^r I_{K_i}^0$, \quad $P_L := \prod_{i = 1}^r P_{K_i}$\quad and \quad $\Cl^0(L) := I_L^0/P_L=\prod_{i = 1}^r \Cl^0(K_i)$.} \] Finally, the norm map $N: \A_L^{\times, 0} \to \A_k^{\times, 0}$ induces a map $N: I_L^0 \to I_k^0$, which we also call norm, and we denote by $I_L^{0(1)}$ its kernel and $P_L^{0(1)} := P_L^0 \cap I_L^{0(1)}$.

\begin{prop} \label{prop: class 0}
With the above notation, we have
\begin{align*}
    h^0(T_{L/k}) = \frac{[I_L^{(1)}:P_L^{(1)}][\mathbb{F}_q^{\times} \cap N(L^{\times}):N(\prod_{i = 1}^r \mathbb{F}_{q_i}^{\times})]}{[U_k \cap N(\A_L^{\times, 0}):N(U_L)]}.
\end{align*}
\end{prop}
\begin{proof}
Consider the two commutative diagrams
\begin{center}
    \begin{tikzcd}
    0 \arrow[r] & U_L \arrow[r] \arrow[d, "N"] & \A_L^{\times, 0} \arrow[r] \arrow[d, "N"] & I_L^0 \arrow[r] \arrow[d, "N"] & 0 \\
    0 \arrow[r] & U_k \arrow[r] & \A_k^{\times, 0} \arrow[r] & I_k^0 \arrow[r] & 0
    \end{tikzcd}
\end{center}
and
\begin{center}
    \begin{tikzcd}
    1 \arrow[r] & \prod\limits_{i = 1}^r \mathbb{F}_{q_i}^{\times} \arrow[r] \arrow[d, "N_{L/k}"] & L^{\times} \arrow[r, "\mathrm{div}"] \arrow[d, "N_{L/k}"] & P_L \arrow[r] \arrow[d, "N"] & 1 \\
    1 \arrow[r] & \mathbb{F}_q^{\times} \arrow[r] & k^{\times} \arrow[r, "\mathrm{div}"] & P_k \arrow[r] & 1
    \end{tikzcd}
\end{center}
where the norm map $N_{L/k} = \prod N_{K_i/k}$ restricted to $\prod_{i = 1}^r \mathbb{F}_{q_i}^{\times}$ is given by
\begin{align*}
    N_{L/k}((x_i))=\prod_{i=1}^r N_{K_i/k}(x_i) = \prod_{i=1}^r \left(N_{\mathbb{F}_{q_i}/\mathbb{F}_q}(x_i)\right)^{[K_i/k]/[\mathbb{F}_{q_i}: \mathbb{F}_q]} \quad \text{ for } x_i \in \mathbb{F}_{q_i}
\end{align*}
and we denote by $\left(\prod_{i = 1}^r \mathbb{F}_{q_i}^{\times} \right)^{(1)}$ its kernel. By the snake lemma, each diagram gives an exact sequence, respectively:
\begin{center}
    \begin{tikzcd}
    0 \arrow[r] & U_L^{(1)} \arrow[r] & \A_L^{0(1)} \arrow[r] & I_L^{(1)} \arrow[r, "\delta_1"] & U_k/N(U_L) \arrow[r] & \A_k^{\times, 0}/N(\A_L^{\times, 0}) \arrow[r] & \dotsb; 
    \end{tikzcd}
\end{center}
\begin{center}
    \begin{tikzcd}[column sep = scriptsize]
    1 \arrow[r] & \left(\prod_{i = 1}^r \mathbb{F}_{q_i}^{\times} \right)^{(1)} \arrow[r] & L^{(1)} \arrow[r] & P_L^{(1)} \arrow[r, "\delta_2"] & \mathbb{F}_q^{\times}/N(\prod_{i = 1}^r \mathbb{F}_{q_i}^{\times}) \arrow[r] & k^{\times}/N(L^{\times}) \arrow[r] & \dotsb.
    \end{tikzcd}
\end{center}
The images $\mathrm{Im}\, \delta_1$ and $\mathrm{Im}\, \delta_2$ are finite abelian groups. Thus we have
\begin{center}
    \begin{tikzcd}
    1 \arrow[r] & \dfrac{\A_L^{\times, 0} \cap \A_L^{(1)}}{U_L^{(1)}} \arrow[r] & I_L^{(1)} \arrow[r] & \dfrac{U_k \cap N(\A_L^{\times, 0})}{N(U_L)} \arrow[r] & 1 \\
    1 \arrow[r] & \dfrac{L^{(1)}}{\left(\prod_{i = 1}^r \mathbb{F}_{q_i}^{\times} \right)^{(1)}} \arrow[r] \arrow[u, hook, "f'"] & P_L^{(1)} \arrow[r] \arrow[u, hook, "f"] & \dfrac{\mathbb{F}_q^{\times} \cap N(L^{\times})}{N(\prod_{i = 1}^r \mathbb{F}_{q_i}^{\times})} \arrow[r] \arrow[u, hook, "{f"}"] & 1.
    \end{tikzcd}
\end{center}
Since $q(f) = q(f') q(f")$,
\begin{align*}
    [I_L^{(1)}: P_L^{(1)}] 
    = [\A_L^{\times, 0} \cap \A_L^{(1)}:L^{(1)} U_L^{(1)}] 
    \frac{[U_k \cap N(\A_L^{\times, 0}):N(U_L)]}{[\mathbb{F}_q^{\times} \cap N(L^{\times}):N(\prod_{i = 1}^r \mathbb{F}_{q_i}^{\times})]}. 
    \  \text{\qed}
\end{align*}
\end{proof}

\begin{lem} \label{lem: describe first term E^O(L/k)}
We have
\begin{align*}
    \frac{[k^{\times} \cap N(\A_L^{\times,0}):N(L^{\times, 0})]}{[\A_k^{\times, 0} : k^{\times} N(\A_L^{\times, 0})]} = \frac{\abs{\Sha(L/k)}}{[L_{ab}:k]} \cdot q(\phi^0),
\end{align*}
where $\phi^0: \A_k^{\times, 0}/N(\A_L^{\times, 0}) \to \A_k^{\times}/N(\A_L^{\times})$ is the canonical homomorphism.
\end{lem}
\begin{proof}
Noting that $\A_k^{\times, 0}$ is the kernel of the degree map $\deg_k: \A_k^{\times} \to \Z$, we have
\begin{align*}
    \ker \phi^0 = \frac{\A_k^{\times, 0} \cap N(\A_L)}{N(\A_L^{\times, 0})}, \quad
    \coker \phi^0 = \frac{\A_k^{\times}}{\A_k^{\times, 0} N(\A_L^{\times})} \simeq \frac{\Z}{\deg_k (N(\A_L^{\times}))}
\end{align*}
and then
\begin{align*}
    q(\phi^0) = \frac{[\Z: \deg_k (N(\A_L^{\times})])}{[\A_k^{\times, 0} \cap N(\A_L^{\times}): N(\A_L^{\times, 0})]}.
\end{align*}
Let $b$ be the canonical homomorphism $b: \A_k^{\times, 0}/\left( k^{\times} N(\A_L^{\times, 0}) \right) \to \A_k^{\times}/\left(k^{\times} N(\A_L^{\times})\right)$. Then we have
\begin{align*}
    \ker b = \frac{k^{\times} \left( \A_k^{\times, 0} \cap N(\A_L^{\times}) \right)}{k^{\times} N(\A_L^{\times, 0})}, \quad
    \coker b = \frac{\A_k^{\times}}{\A_k^{\times, 0} k^{\times} N(\A_L^{\times})} \simeq \frac{\Z}{\deg_k (N(\A_L^{\times}))}.
\end{align*}
From the finiteness of domain and codomain of $b$ and the exact sequence
\begin{center}
    \begin{tikzcd}
    1 \arrow[r]
    & \dfrac{k^{\times} \cap N(\A_L^{\times})}{k^{\times} \cap N(\A_L^{\times, 0})} \arrow[r]
    & \dfrac{\A_k^{\times, 0} \cap N(\A_L^{\times})}{N(\A_L^{\times, 0})} \arrow[r]
    & \ker b \arrow[r]
    & 1
    \end{tikzcd}
\end{center}
we compute that
\begin{align*}
    \frac{[\A_k^{\times}: k^{\times} N(\A_L^{\times})]}{[\A_k^{\times, 0}: k^{\times} N(\A_L^{\times, 0})]} 
    &= q(b) 
    = \frac{[\Z: \deg_k (N(\A_L^{\times}))]}{\abs{\ker b}} \\
    &\overset{(*)}{=} q(\phi^0) \cdot [k^{\times} \cap N(\A_L^{\times}): k^{\times} \cap N(\A_L^{\times, 0})] \\
    &= q(\phi^0) \cdot \frac{[k^{\times} \cap N(\A_L^{\times}): N(L^{\times})]}{[k^{\times} \cap N(\A_L^{\times, 0}): N(L^{\times})]}.
\end{align*}

In $(*)$ we use the fact that for abelian subgroups $A$, $B$, and $C$ of an abelian group $G$ with $B \subset A$, we have
\begin{align*}
    [A:B] = [A \cap C: B \cap C][AC:BC].
\end{align*}
This follows from considering the exact sequences
\begin{center}
    \begin{tikzcd}
    0 \arrow[r]
    & (A \cap BC)/B \arrow[r]
    & A/B \arrow[r]
    & AC/BC \arrow[r]
    & 0;
    \end{tikzcd}
    \begin{tikzcd}
    0 \arrow[r]
    & B \cap C \arrow[r]
    & A \cap C \arrow[r]
    & (A \cap BC)/B \arrow[r]
    & 0.
    \end{tikzcd}
\end{center}
Finally, using Theorem \ref{thm: L_ab over k} we obtain
\begin{align*}
    \frac{[k^{\times} \cap N(\A_L^{\times,0}):N(L^{\times})]}{[\A_k^{\times, 0} : k^{\times} N(\A_L^{\times, 0})]} 
    = q(\phi^0) \cdot \frac{[k^{\times} \cap N(\A_L^{\times}): N(L^{\times})]}{[\A_k^{\times}: k^{\times} N(\A_L^{\times})]}
    = q(\phi^0) \cdot \frac{\abs{\Sha(L/k)}}{[L_{ab}:k]}. \ \text{\qed}
\end{align*}
\end{proof}

\begin{thm}\label{thm:E0Lk}
Let the notation be as above. We have
\begin{align*}
    E^0(L/k) = \frac{\abs{\Sha(L/k)}}{[L_{ab}: k]} \cdot q(\phi^0) \cdot \frac{[U_k : N(U_L)]}{[\mathbb{F}_q^{\times} : N(\prod_i \mathbb{F}_{q_i}^{\times})]}.
\end{align*}
\end{thm}
\begin{proof}
Let $a: k^{\times}/N(L^{\times}) \to \A_k^{\times, 0}/N(\A_L^{\times, 0})$ be the canonical homomorphism. Consider the commutative diagram
\begin{center}
    \begin{tikzcd}
    1 \arrow[r] & \mathbb{F}_q^{\times} \arrow[r] \arrow[d, "{a'}"] & k^{\times}/N(L^{\times}) \arrow[r] \arrow[d, "a"] & k^{\times}/\mathbb{F}_q^{\times} N(L^{\times}) \arrow[r] \arrow[d, "{a''}"] & 1 \\
    1 \arrow[r] & U_k/U_k \cap N(\A_L^{\times, 0}) \arrow[r] & \A_k^{\times, 0}/N(\A_L^{\times, 0}) \arrow[r] & \A_k^{\times, 0}/U_k N(\A_L^{\times, 0}) \arrow[r] & 1
    \end{tikzcd}
\end{center}

From the commutative diagram
\begin{center}
    \begin{tikzcd}
    1 \arrow[r]
    & \mathbb{F}_q^{\times} \cap N(L^{\times})/N(\prod_i \mathbb{F}_{q_i}^{\times}) \arrow[r] \arrow[d, "{b'}"]
    & \mathbb{F}_q^{\times}/N(\prod_i \mathbb{F}_{q_i}^{\times}) \arrow[r] \arrow[d, "{b}"]
    & \mathbb{F}_q^{\times}/\mathbb{F}_q^{\times} \cap N(L^{\times}) \arrow[r] \arrow[d, "{a'}"]
    & 1 \\
    1 \arrow[r]
    & U_k \cap N(\A_L^{\times, 0})/N(U_L) \arrow[r]
    & U_k/N(U_L) \arrow[r]
    & U_k/U_k \cap N(\A_L^{\times, 0}) \arrow[r]
    & 1
    \end{tikzcd}
\end{center}
we have
\begin{align*}
    q(a') 
    = \frac{q(b)}{q(b')}
    = \frac{[U_k : N(U_L)]}{[\mathbb{F}_q^{\times} : N(\prod_i \mathbb{F}_{q_i}^{\times})]} 
    \cdot \frac{[\mathbb{F}_q^{\times} \cap N(L^{\times}) : N(\prod_i \mathbb{F}_{q_i}^{\times})]}{[U_k \cap N(\A_L^{\times, 0}) : N(U_L)]}.
\end{align*}
Now the diagram
\begin{center}
    \begin{tikzcd}
    1 \arrow[r]
    & P_L \arrow[r] \arrow[d, "{N'}"]
    & I_L^0 \arrow[r] \arrow[d, "{N}"]
    & \Cl^0(L) \arrow[r] \arrow[d, "{N''}"]
    & 1 \\
    1 \arrow[r]
    & P_k \arrow[r]
    & I_k^0 \arrow[r]
    & \Cl^0(k) \arrow[r]
    & 1
    \end{tikzcd}
\end{center}
induces the exact sequence
\begin{center}
    \begin{tikzcd}
    I_L^{0 (1)} \arrow[r]
    & \ker N'' \arrow[r, "\delta"]
    & P_k/N'(P_L) \arrow[r, "c"]
    & I_k^0/N(I_L^0) \arrow[r]
    & \coker N'' \arrow[r]
    & 1 \\
    &
    & k^{\times}/ \mathbb{F}_q^{\times} \cdot N(L^{\times}) \arrow[r, "{a''}"] \arrow[u, "\wr"]
    & \A_k^{\times, 0}/U_L \cdot N(\A_L^{\times, 0}) \arrow[u, "\wr"]
    &
    &
    \end{tikzcd}
\end{center}
Then
\begin{align*}
    q(a'') &= q(c) = \frac{\abs{\coker N''}}{\abs{\ker c}} = \abs{\coker N''} \cdot \frac{[I_L^{0(1)}: P_L]}{\abs{\ker N''}} \\
    &= q(N'') [I_L^{0(1)}: P_L] = [I_L^{0(1)}: P_L] \cdot \frac{h^0(k)}{h^0(L)}.
\end{align*}
Using Proposition \ref{prop: class 0}, we have
\begin{align*}
    q(a) = q(a') q(a'') = \frac{[U_k : N(U_L)]}{[\mathbb{F}_q^{\times} : N(\prod_i \mathbb{F}_{q_i}^{\times})]} \cdot E^0(L/k)^{-1}.
\end{align*}
On the other hand from the definition of $q(a)$,
\begin{align*}
    q(a) = \frac{[\A_k^{\times, 0}: k^{\times}  N(\A_L^{\times, 0})]}{[k^{\times} \cap N(\A_L^{\times, 0}): N(L^{\times})]}.
\end{align*}
Using Lemma \ref{lem: describe first term E^O(L/k)}, we obtain the formula
\begin{align*}
    E^0(L/k) = \frac{\abs{\Sha(L/k)}}{[L_{ab}: k]} \cdot q(\phi^0) \cdot \frac{[U_k : N(U_L)]}{[\mathbb{F}_q^{\times} : N(\prod_i \mathbb{F}_{q_i}^{\times})]}. \ \text{\qed}
\end{align*}
\end{proof}

\section{Exploration of terms in Theorem~\ref{thm:ESLk} and examples}\label{sec:E}
\def\ab{{\rm ab}}
\def\ur{{\rm ur}}
In this section we explore the terms appearing in the class number relation arising from multinorm-one tori (Theorem~\ref{thm:ESLk}).

\subsection{Local norm indices ${[U_{k,S}:N(U_{L,S})]}$}  

\subsubsection{}
Suppose that $K/k$ is a finite separable extension of local fields. We denote by $K_{\ab}$ and $K_{\ur}$ the maximal abelian and unramified subextensions of $K/k$, respectively. 
%Let $e$ and $f$ be the ramification index and the residue degree of $K_\ab$ over $k$.
Local class field theory says that $N_{K/k}(K^\times)=N_{K_{\ab}/k} (K_{\ab}^\times)$ and we have a commutative diagram 
\[ 
\begin{tikzcd}
0 \arrow[r] & O_k^\times/N(O_{K_\ab}^\times) \arrow[r] \arrow[d, "\wr"'] &  k^\times/N(K_\ab^\times) \arrow[r]\arrow[d, "{\mathrm{Art}}", "\wr"'] & \Z/f\Z \arrow[r] \arrow[d, "\wr"']  & 0 \\
0 \arrow[r] & \Gal(K_\ab/K_{\ab,\ur}) \arrow[r] & \Gal(K_\ab/k) \arrow[r] & \Gal(\lambda/\kappa) \arrow[r]  & 0 
\end{tikzcd}
\]
where $\lambda$ and $\kappa$ are the residue fields of $K_\ab$ and $k$, respectively, and $f=[\lambda:\kappa]$ is the residue degree of $K_\ab$ over $k$. In particular, we have 
\begin{equation}
    [O_k^\times: N(O_K^\times)]=e(K_\ab/k),     
\end{equation}
the ramification index of $K_\ab$ over $k$.

\subsubsection{}
Recall \[ U_{k,S}=\prod_{v\in S} k_v^\times \times 
\prod_{v\not\in S} O_v^\times \quad \text{and} \quad U_{L,S}=\prod_{v\in S} L_v^\times \times \prod_{v\not\in S} O_{L_v}^\times,\  \text{where} \ L_v=\prod_{w|v} L_w, \] 
and we have
\[ [U_{k,S}:N(U_{L,S}]=\prod_{v\in S} [k_v^\times:N(L_v^\times)] \cdot \prod_{v\not\in S} [O_v^\times:N(O_{L_v}^\times)]. \]

Fix a separable closure $k_{v,s}$ of $k_v$ and let $k_v^\ab$ and $k_v^\ur$ be the maximal abelian and unramified extensions of $k_v$ in $k_{v,s}$, respectively. One has $\Gal(k_v^\ab/k_v^\ur)\simeq O_v^\times$. 
For each place $w|v$ of $L$, we choose an embedding $\iota:L_w\embed k_{v,s}$ over $k_v$. Since $L_{w,\ab}$ is Galois over $k_v$, its image $\iota(L_{w,\ab})$ is independent of the choice of $\iota$, which we denote by $L_{w,\ab}$ again for simplicity. 
Put 
\begin{equation}
    \label{eq:Lvab}
   L_{v,\ab} := \bigcap_{w|v} L_{w,\ab} 
\end{equation}
and 
\begin{equation}
    \label{eq:tLvab}
    \wt L_{v,\ab}:=\text{the compositum of the abelian extensions $L_{w,\ab}$ of $k_v$ for all $w|v$}.
\end{equation}
If $v$ is a finite place of $k$, we choose a finite unramified extension $k_v'$ of $k_v$ which contains the maximal unramified subextension $(\wt L_{v,\ab})_\ur$ of $\wt L_{v,\ab}/k_v$ and set
\begin{equation}\label{eq:Lpvab}
           L'_{w,\ab}:=k_v' L_{w,\ab} \quad \text{and} \quad L'_{v,\ab}:=\bigcap_{w|v} L'_{w,\ab}. 
\end{equation}
The degree $[L'_{v,\ab}:k_v']$ does not depend on the choice of $k_v'$. If there is a place $w|v$ of $L$ which is unramified in $L/k$, then $[L'_{v,\ab}:k_v']=1$. For this reason, we 
write
\begin{equation} \label{eq:ev}
    e_v(L/k):=[L'_{v,\ab}:k_v']. 
\end{equation}
When $L_v$ is an abelian field extension of $k_v$, the invariant $e_v(L/k)$ coincides with the ramification index of $v$ in $L/k$.

%denote by 

\begin{prop}\label{prop:4.1}
    \begin{enumerate}
        \item For any place $v$, we have $[k_v^\times:N(L_v^\times)]=[L_{v,\ab}:k_v].$
        \item For any finite place $v$, 
        we have $[O_v^\times: N(O_{L_v}^\times)]=[L'_{v,\ab}:k_v']$.
 %       let $k_v'$ be a finite unramified extension of $k_v$ which contains the maximal unramified extension $(\wt L_{v,\ab})_\ur$ of $k_v$ in $\wt L_{v,\ab}$, and let
 %       \begin{equation}
 %          L'_{w,\ab}:=k_v' L_{w,\ab}, \quad L'_{v,\ab}:=\bigcap_{w|v} L'_{w,\ab}. 
 %       \end{equation}
 %       Then $[O_v^\times: N(O_{L_v}^\times)3]=[L'_{v,\ab}:k_v']$.
    \end{enumerate}
\end{prop}
\begin{proof}
(1) By local class field theory, we have
\begin{equation}
    N_{L_{v,\ab}/k_v}(L_{v,\ab}^\times)=\prod_{w|v} N_{L_{w,\ab}/k_v}(L_{w,\ab}^\times)=\prod_{w|v} N_{L_w/k_v}(L_w^\times)=N(L_v^\times) 
\end{equation}
and get 
\begin{equation}
    [L_{v,\ab}:k_v]=[k_v^\times:N_{L_{v,\ab}/k_v}(L_{v,\ab}^\times)]=[k_v^\times:N(L_v^\times)].
\end{equation}   

(2) 
The open subgroup corresponding to $L_{w,\ab}^\ur=k_v^\ur L_{w,\ab}$ is $U_w:=N_{L_w/k_v}(O_{L_w}^\times)$. If we put $U_v=\cap_{w|v} U_w$, then it corresponds to the field extension $(\wt L_{v,\ab})^\ur$. We have
\[ \Gal((\wt L_{v,\ab})^\ur/L_{w,\ab}^\ur)\simeq \Gal((\wt L_{v,\ab})'/L_{w,\ab}')\simeq U_w/U_v. \]
and hence 
\[ \Gal((\wt L_{v,\ab})'/L'_{v,\ab})\simeq \prod_{w|v} U_w/U_v=\left (\prod_{w|v} U_w \right )/U_v=N(O_{L_v}^\times)/U_w.\]
Therefore, 
\[ [L'_{v,\ab}:k_v']=[O_v^\times: N(O_{L_v}^\times)]. \  \text{\qed}\]
\end{proof}

\begin{rem} In Proposition~\ref{prop:4.1} we make an auxiliary base change of $L_v$ by a sufficiently large unramified extension $k_v'$ so that $\Gal(L_{w,\ab}'/k_v')\simeq [O_v\times: N_{L_w/k_v}(O_{L_w}^\times)]$ and take the intersection $\cap_{w|v} L'_{w,\ab}$. This is necessary even when each field extension $L_w/k_v$ is abelian and totally ramified. Indeed, one has $\Gal(L_w/k_v)\simeq O_v^\times/N_{L_w/k_v}(O_{L_w}^\times)$ and 
\[ \Gal(L_{v,ab}/k_v)\simeq O_v^\times/(N(L_v^\times)\cap O_v^\times).\]
However, we only have the inclusion $N(O_{L_v}^\times)\subset (N(L_v^\times)\cap O_v^\times)$ and no equality in general. 

For example, let $k_v=\Q_\ell$ where $\ell$ is an odd prime. 
By local class field theory, there are exactly two ramified quadratic fields $L_{w_1}$ and $L_{w_2}$. By weak approximation, there exist quadratic extensions $K_1$ and $K_2$ such that $(K_1)_v \simeq L_{w_1}$ and $(K_2)_v \simeq L_{w_1}$. Put $L=K_1\times K_2$ and hence $L_v=L_{w_1}\times L_{w_2}$. We have $N(L_v^\times)=k_v^\times$ and $N(L_v^\times)\cap O_v^\times=O_v^\times$, while $N(O_{L_v}^\times)\subset O_v^\times$ is of index two.  
\end{rem}

\subsection{The global unit norm index $[O_{k, S}^{\times}: N(O_{L, S}^{\times})]$}

\begin{prop}\label{prop:4.2}
Let $F_L$ be a free part of $O_{L,S}^\times$, that is, $F_L$ is a finitely generated free subgroup such that $O_{L, S}^\times = \mu_L \times F_L$, where $\mu_L = \prod \mu_{K_i}$ is the group of roots of unity in $L$.

\begin{enumerate}
    \item We can choose a free part $F_k$ of $O_{k, S}^\times$ such that $O_{k, S}^{\times} = \mu_k \times F_k$ and $N(F_L) \subset F_k$. Then we have $O_{k, S}^{\times}/N(O_{L, S}^{\times}) \simeq \mu_k/N(\mu_L) \times F_k/N(F_L)$.

    \item Moreover, let $\{\xi_j\}$ be a system of fundamental units of $O_{L, S}^{\times}$ and let $\pi_d: F_k \to F_k/(F_k)^d$ be the natural projection, where $d$ is the greatest common divisor of $[K_i: k]$, $1 \leq i \leq r$. Then $F_k/N(F_L) \simeq (F_k/(F_k)^d)/E_L$, where $E_L$ is the subgroup generated by $\pi_d(\xi_j)$ for all~$j$.
\end{enumerate}
\end{prop}
\begin{proof}
    (1) 
    % We first explain the existence of such a group $F_k$. 
    Note that $N(F_L)$ is a finitely generated free subgroup of $O_{k, S}^\times$. By Zorn's lemma, there exists a maximal finitely generated free subgroup $F_k$ of $O_{k, S}^\times$ such that $N(F_L) \subset F_k$. Then 
    \[O_{k, S}^{\times}/N(O_{L, S}^{\times}) 
    = \frac{\mu_k \times F_k}{N(\mu_L) \times N(F_L)} 
    \simeq \mu_k/N(\mu_L) \times F_k/N(F_L).\]

    (2) Note that the group $O_{L,S}^\times=\prod_{i=1}^r O_{K_i,S}^\times$ contains the subgroup $\prod_{i=1}^r O_k^\times$ and hence the free subgroup $F_L$ contains the free subgroup $\prod_{i=1}^r F_k$. Therefore, we have  
    \begin{equation}\label{eq:NFL}
      N(F_L) \supset N\left(\prod_{i=1}^r F_k\right)= F_k^{[K_1:k]} \cdots F_k^{[K_r:k]}=F_k^d,
    \end{equation}
    and $F_k/N(F_L)$ is the quotient of the $F_k/F_k^d$ by its image of $N(F_L)$. This shows statement~(2). \qed
\end{proof}

It is rather easy to compute the index of torsion part $[\mu_k:N(\mu_L)]$ for each given specific case. However, finding a system of fundamental units is a classical problem in number theory, which is known to be difficult in general. However, by Proposition~\ref{prop:4.2}, if the degrees $[K_i:k]$ have only the 
trivial common divisor, then $O_{k, S}^{\times}/N(O_{L, S}^{\times})=\{1\}$,  cf.~\eqref{eq:NFL}.

If $k=\Q$ and $S$ is the set of archimedean places, then we have that $[O_{k, S}^\times:N(O_{L, S}^\times)]=[\Z^\times:N(O_{L}^\times)]$ is $1$ or $2$. Moreover, 
\[ [O_{k, S}^\times:N(O_{L, S}^\times)]=1 \iff -1\in N_{K_i/k} (O_{K_i}^\times) \text{ for some } i. \]
In particular, if one of the degrees $[K_i:k]$ is odd, then $[O_k^\times:N(O_L^\times)]=1$. When $K$ is a real quadratic field, one can use continued fractions to compute the fundamental unit $\epsilon_K$ (cf.~\cite{olds:63}) and determine whether $N(\epsilon_K)=-1$. In the case where $K=\Q(\sqrt{p})$ and $p$ is a prime, we know (from \cite[Lemma 2.4]{xue-yang-yu:num_inv}) that
\begin{align}\label{eq:fu}
    N(\epsilon_K) = 
    \begin{cases}
    1, & \text{if $p \equiv 3\! \pmod{4}$}; \\
    -1, & \text{otherwise}.
    \end{cases}
\end{align}

However, there is no known direct determination of $N(\epsilon_K)$ 
from the discriminant of $K$ as \eqref{eq:fu} in general. 

\subsection{The Tate-Shafarevich group $\Sha(L/k)$}

The most interesting and involved term in the class number formula for  multinorm-one tori is $|\Sha(L/k)|$.
In this section we organize some results about computation of $\Sha(L/k)$. In particular, we introduce a criterion for $\Sha(L/k) = 0$. 
%For any finite commutative group $G$, let $G^\vee:=\Hom(G, \Q/\Z)$ denote the Pontryagin dual. 

First, we have an exact sequence of $k$-tori
\begin{center}
    \begin{tikzcd}
    1 \arrow[r] & T_{L/k} \arrow[r, "\iota"] & R_{L/k} \mathbb{G}_{m, L} \arrow[r, "N"] & \mathbb{G}_{m, k} \arrow[r] & 1.
    \end{tikzcd}
\end{center}
By Hilbert's Theorem 90, we have an isomorphism 
\begin{equation}\label{eq:4.1}
   k^\times/N_{L/k}(L^\times) \xrightarrow[]{\,\sim\,} H^1(k, T_{L/k}). 
\end{equation}
Taking the kernel of the local-global map on each side, 
we have the isomorphisms
\begin{equation}\label{eq:4.2}
   \Sha(L/k) \simeq \Sha^1(k, T_{L/k}) \simeq \Sha^2(k, \widehat{T}_{L/k})^{\vee}, 
\end{equation}
where the second isomorphism is given by Poitou-Tate duality \cite[Theorem 6.10]{platonov-rapinchuk}.

%(Thinking whether we will use this information)
%In \cite{BLP19}, we see that the formula of $\Sha(L/k)$, when $K_1, \dotsc, K_r$ are cyclic extensions over a number field $k$, is still complicated. 
%In this situation, \cite[Proposition 8.6]{BLP19} asserts that
%\begin{align*}
%    \Sha(L/k) \simeq \bigoplus\limits_{p} \Sha(L(p)/k), \quad L(p) = \prod\limits_{i = 1}^r K_i(p)
%\end{align*}
%where $p$ runs through prime numbers dividing $[K_1:k] \dotsb [K_r:k]$ and $K_i(p)$ is the largest subfield of $K_i$ such that $[K_i(p):k]$ is a power of $p$.

Let 
\begin{equation}\label{eq:4.3}
    \Sha^i_{\omega}(k, \wh T):=\left\{[C]\in H^i(k, \wh T)\, |\, [C]_v=0 \text{  for almost all places $v$ of $k$ } \right\},
\end{equation}
where $[C]_v$ denotes the class in $H^i(k_v, \wh T)$ under the restriction map $H^i(k, \wh T)\to H^i(k_v, \wh T)$. 
We shall utilize the following exact sequence
\begin{align*}
    0 \longrightarrow \Sha^2(k,\widehat{T}) \longrightarrow \Sha^2_{\omega}(k,\widehat{T}) \longrightarrow A(T)^{\vee} \longrightarrow 0,
\end{align*}
where $T(k) \hookrightarrow \prod_v T(k_v)$ embeds diagonally and $A(T) := \prod_v T(k_v)/\overline{T(k)}$ is the group measuring the defect of weak approximation of $T$, or its dual version
\begin{equation}
    \label{eq:4.13}
        0 \longrightarrow A(T) \longrightarrow \Sha^2_{\omega}(k,\widehat{T})^{\vee} \longrightarrow \Sha^2(k, \widehat{T})^{\vee} \longrightarrow 0.
\end{equation}
Here $\Sha^2(k,\widehat{T})^\vee \simeq \Sha^1(k,T)$ and $\Sha^2_{\omega}(k, \widehat{T}) \simeq H^1(k, \mathrm{Pic}(\overline{T}))$, where $\overline{T}$ is a smooth compactification of $T$, by \cite[Theorem 6]{voskresenskii1970birational}.
By definition, the local-global principle for $T$ holds if $\Sha^1(k, T)=0$, and  \emph{weak approximation} for $T$ holds if $A(T) = 0$. 
From the second exact sequence \eqref{eq:4.13} we see that
\begin{align*}
    \Sha^2_{\omega}(k,\widehat{T}) = 0 \Longleftrightarrow A(T) = 0 \text{ and } \Sha^1(k,T) = 0.
\end{align*}

Recall that \emph{Hasse norm principle} (HNP) holds for the \'etale $k$-algebra $L/k$ if $\Sha(L/k)=0$ or $\Sha^2(k, \widehat T_{L/k}) = 0$ by \eqref{eq:4.2}.

In the following, we collect some results from the literature.
For a finite separable extension $K/k$ of global fields, we denote by $K^c$ the Galois closure of $K$ over $k$.
We also write $\Sha_{\omega}(L/k)$ for $\Sha^1_\omega(k, T_{L/k})$.

\begin{thm}\label{thm:4.4}
Let $L=\prod_{i=1}^r K_i$ be an \'etale $k$-algebra and $T_{L/k}$ the associated multinorm-one torus.  
The Hasse Norm Principle (HNP) holds for $L/k$, that is, $\Sha(L/k) = 0$, if one of the following conditions holds. 

\noindent\emph{Case 1. Let $r = 1$ and $L = K$.}
\begin{enumerate}[label = (1.\alph*)]
    \item $K/k$ is Galois with Galois group $G$ such that $\Sha^3(G, \Z) = 0$.
    \item $[K:k]$ is a prime.
    \item $[K:k] = n$ and $\Gal(K^c/k)$ is isomorphic to the dihedral group $D_n$.
    \item $[K:k] = n$ and $\Gal(K^c/k)$ is isomorphic to the symmetric group $S_n$.
    \item $[K:k] = n$ and $\Gal(K^c/k)$ is isomorphic to the alternating group $A_n$, for $n \geq 5$.
\end{enumerate}
\emph{Case 2. Let $r = 2$ and $L = K_1 \times K_2$.} We set $F = K_1 \cap K_2$ and let $\Sha_\omega(F/k)=\Sha_\omega^1(k, T_{F/k})$ denote the quotient of the subgroup of $k^\times$ which is a norm locally almost everywhere, by $N_{F/k}(F^\times)$.
\begin{enumerate}[label = (2.\alph*)]
    \item $K_1$ is a cyclic extension and $K_2$ is an arbitrary finite separable extension.
    \item $K_1^c \cap K_2^c = k$.
    \item $K_1$ and $K_2$ are abelian extensions over $k$, and $\Sha(F/k) = 0$.
    \item $\Sha_{\omega}(F/k) = 0$.
\end{enumerate}
\emph{Case 3. General $r\ge 2$.}
\begin{enumerate}[label = (3.\alph*)]
    \item $K_1, \dotsc, K_r$ are Galois over of $k$, the field $K_1 \dotsb K_i \cap (K_{i+1} \dotsb K_r)$ equals to the intersection $F := \bigcap_{i = 1}^r K_i$ for some $1 \leq i \leq r-1$, and $\Sha_{\omega}(F/k) = 0$.
    \item $K_1, \dotsc, K_r$ are Galois over of $k$, and $K_1 \dotsb K_i \cap (K_{i+1} \dotsb K_r) = k$ for some $1\le i\le r-1$.
    \item $K_1, \dotsc, K_r$ are distinct extensions over $k$ of degree $p$, where $p$ is a prime, with $K_i$ is cyclic for some $i$, and either the composition $\wt F := K_1 \dotsb K_r$ has degree $> p^2$ over $k$ or one local degree of $\wt F$ is $> p$.
\end{enumerate}
\end{thm}
\begin{proof}
\emph{Case 1. $r = 1$ and $L = K$.}
\begin{enumerate}[label = (1.\alph*)]
    \item This is a theorem of Tate, cf.~\cite[Theorem 6.11]{platonov-rapinchuk}.
    \item See Bartels \cite[Lemma 4]{bartels1981179}, cf.~\cite[Proposition 6.3]{platonov-rapinchuk}.
    \item See Bartels \cite[Satz 1]{bartels1981arithmetik}.
    \item This is a result of B.~Kunyavskii and V.~Voskresenskii; see \cite{voskresenskii}, cf.~\cite[p.~2]{macedo2020500}.
    \item See Macedo \cite[Theorem 1.1]{macedo2020500}.
    % \item[]  We remark the references for (1.a), (1.b), (1.d) and (1.e) are stated for number fields. However, since the methods in loc.~cit are using Galois cohomology and group theory, they also apply to the function field case.    
\end{enumerate}
We remark the references for (1.a), (1.b), (1.d) and (1.e) are stated for number fields. However, since the methods in loc.~cit use Galois cohomology and group theory, the proofs also apply to the global function field case.  

\noindent\emph{Case 2. $r = 2$ and $L = K_1 \times K_2$.}
\begin{enumerate}[label = (2.\alph*)]
    \item The case where $K_2/k$ is Galois is  proved by H\"urlimann in \cite[Proposition 3.3]{hurlimann} and the general case is proved in \cite[Proposition 4.1]{BLP19}.
    \item Pollio and Rapinchuk proved that this condition implies $\Sha(L/k) = 0$ in \cite{pollio-rapinchuk}.
    \item In \cite{pollio}, Pollio proved that if $K_1$ and $K_2$ are abelian extensions of $k$, then $\Sha(L/k) = \Sha(F/k)$.
    \item This follows from Demarche and Wei's work~\cite{demarche-wei}. Applying \cite[Theorem 6]{demarche-wei} to the case $I = \{1\}$ and $J = \{2\}$, we obtain $\Sha_{\omega}(L/k) = \Sha_{\omega}(F/k)$. In particular, $\Sha_{\omega}(L/k) = 0$ implies $\Sha(L/k) = 0$.
\end{enumerate}

\noindent\emph{Case 3. General $r\ge 2$.} 
\begin{enumerate}[label = (3.\alph*)]
    \item In \cite[Example 9]{demarche-wei}, Demarche and Wei proved that if $K_1, \dotsc, K_r$ are Galois extensions of $k$ and $(K_1 \dotsb K_i) \cap (K_{i+1} \dotsb K_r) = F = \bigcap_{i = 1}^r K_i$ for some $1 \leq i \leq r$, then $\Sha_{\omega}(L/k) \simeq \Sha_{\omega}(F/k)$.
    \item This condition originates from \cite[Theorem 1]{demarche-wei}.
    \item This is an application of Bayer-Fluckiger, T.-Y. Lee and Parimala's \cite[Proposition 8.5]{BLP19}. Note that when $r = 2$ this recovers condition (2.a). \qed 
\end{enumerate}
\end{proof}

\begin{rem}\label{rem:morishita}
   Theorem~\ref{thm:4.4}(2.b) implies that the term $|\Sha_k(T')|$ in the main theorem of \cite[p.~135]{morishita:nagoya1991} is equal to $1$.
\end{rem}

\begin{cor} \label{cor:refine}
Let the notation be as in Theorem~\ref{thm:ESLk} and in Proposition~\ref{prop:4.1}. Assume one of the conditions in Theorem~\ref{thm:4.4} holds. 
Then 
\begin{equation}
    \label{eq:ESLk_refine}
   E_S(L/k) 
   = \frac{\prod_{v\in S}[L_{v,\ab}:k_v] \cdot \prod_{v\in R(L/k)\setminus S} e_v(L/k)} {[L_{ab}:k]\cdot [ {O}_{k, S}^{\times}: N( {O}_{L, S}^{\times})]}, 
%   \cdot \frac{[U_{k, S}: N(U_{L, S})]}{[ {O}_{k, S}^{\times}: N( {O}_{L, S}^{\times})]}
  \end{equation}
 
  \begin{equation}
    \label{eq:ES+Lk_refine}
   E_S^+(L/k) 
   = \frac{\prod_{v\in S}[L_{v,\ab}:k_v] \cdot \prod_{v\in R(L/k)\setminus S} e_v(L/k)}{[L_{ab}:k] \cdot q(\phi)\cdot [ {O}_{k, S}^{\times +}: N( {O}_{L, S}^{\times +})]}, 
%   \cdot \frac{[U_{k, S}: N(U_{L, S})]}{[ {O}_{k, S}^{\times +}: N( {O}_{L, S}^{\times +})]}
  \end{equation}
and 
   \begin{equation}
    \label{eq:ES0Lk_refine}
   E^0(L/k) 
   = \frac{q(\phi^0)\cdot \prod_{v\in R(L/k)} e_v(L/k)}{[L_{ab}:k] \cdot [\mathbb{F}_q^{\times} : N(\prod_i \mathbb{F}_{q_i}^{\times})] }, 
  \end{equation} 
where $L_{v,\ab}$ is a finite abelian extension of $k_v$ defined in \eqref{eq:Lvab}, $e_v(L/k)$ is defined in \eqref{eq:ev}, and $R(L/k)$ is the finite set of finite places $v$ of $k$ for which none of the places $w|v$ of $L$ is unramified in $L/k$.
%$L_{ab}$ is the maximal abelian extension of $k$ contained in all $K_i$ and $\phi: k^{\times +}/N L^{\times +} \to k^{\times} / N L^{\times}$ is the canonical homomorphism.
\end{cor}
\begin{proof}
   This follows from Theorems~\ref{thm:ESLk} and \ref{thm:E0Lk}, Proposition~\ref{prop:4.1} and Theorem~\ref{thm:4.4}. \qed
\end{proof}

%\subsection{On computing $\Sha(L/k)$} 
As in Theorem~\ref{thm:4.4}, there have already been many affirmative results for determining the HNP for $L/k$. However, when the HNP for $L/k$ fails, results for computing $\Sha(L/k)$ are only sporadic. 

For $r=1$ and $K/k$ a Galois extension with group $G$, a theorem of Tate gives us a general method for computing $\Sha(K/k)$ through the canonical isomorphism 
\[ \Sha(K/k)\simeq \Sha^3(G,\Z).\]
Via the natural isomorphism $H^3(G,\Z)\simeq \Hom(H_2(G,\Z), \Q/\Z)$, this reduces the problem to computing the Schur multiplier $M(G)\simeq H_2(G,\Z)$ of $G$ and computing the cokernel of the map
\[ \bigoplus_{v} M(G_v) \to M(G), \]
where $v$ runs through all places of $k$ ramified in $K$ and $G_v$ denotes the ramification group of $v$. 

For higher $r$, Bayer-Fluckiger, Lee and Parimala \cite{BLP19} made a breakthrough for computing $\Sha(L/k)$ in the case where one factor of $L$ is cyclic over $k$. Morever, when every factor of $L$ is a cyclic extension of $k$, the authors gave a necessary and sufficient condition for $\Sha(L/k)=0$ under a mild condition. Extending the work \cite{BLP19}, Lee~\cite{lee2022tate} gave a general formula for computing $\Sha(L/k)$ when all factors of $L$ have $p$-power degrees. Combining Lee's result and a reduction result \cite[Proposition 8.6]{BLP19}, the group $\Sha(L/k)$ is essentially known when all factors of $L$ are cyclic.

As the last part of this article, we present briefly Lee's formula for $\Sha(L/k)$, and describe a further result \cite{HHLY} for computing a certain class of $\Sha(L/k)$.
Let us write $L=\prod_{i=0}^m K_i$ and assume that each $K_i/k$ is cyclic and that $\cap_{i=0}^r K_i=k$.
%At the end we address a meth*od for computing $\Sha(L/k)$ in the works of Bayer-Fluckiger, T.-Y. Lee and Parimala \cite{BLP19} and of T.-Y. Lee \cite{lee2022tate}. 
%Assume that each factor $K_i$ of the \'etale $k$-algebra $L=\prod_{i=0}^m {K_i}$ is cyclic and that $\cap_{i=0}^r K_i=k$. 
By \cite[Proposition 8.6]{BLP19}, each $p$-primary subgroup $\Sha(L/k)(p)$ is isomorphic to $\Sha(L(p))$, where $L(p)$ is the maximal \'etale $k$-subalgebra of $L$ of $p$-power degree. Thus, without loss of generality we may assume that each $K_i/k$  has $p$-power degree, say degree $p^{\epsilon_i}$. %For simplicity let us also assume that  ${\cap}_{i=0}^{m} K_{i}=k$ and $\epsilon_0=\min_{0\le i\le m}\{\epsilon_i\}$. 

%In what follows, we let $L=\prod_{i=0}^{m} K_i$, where $K_i$ are \emph{cyclic extensions} of $k$ of degree $p^{\epsilon_{i}}$ for a positive integer $\epsilon_{i}$. Assume ${\cap}_{i=0}^{m} K_{i}=k$ and $\epsilon_0=\min_{0\le i\le m}\{\epsilon_i\}$. %Let $\mathcal{I}=\{1,...,m\}$, 
For any $0\le i,j\le m$, 
% in\mathcal{I}:=\{0,1,\dots,m\}$, 
set
\begin{enumerate}
    \item[(i)]
    $p^{e_{i,j}}=[K_{i}\cap K_{j}:k]$, and
    \item[(ii)]
    $e_{i}=\epsilon_{0}-e_{0,i}$.
\end{enumerate}
We may assume that $e_{i}\geq e_{i+1}$ and that $\epsilon_0=\min_{0\le i\le m}\{\epsilon_i\}$. For $0\leq r\leq\epsilon_{0}$, set
\[ \text{$U_{r}:=\{i\in\mathcal{I}|\text{ }e_{0,i}=r\}$}. \]
%\quad $U_{>r}:=\{i\in\mathcal{I}|\text{ }e_{0,i}>r\}$,\quad and \quad $U_{<r}:=\{i\in\mathcal{I}|\text{ }e_{0,i}<r\}$.} \]
\begin{defn}
(1) Let $i,j \in \mathcal{I}:=\{1,\dots, m\}$ and $l$ be a nonnegative integer. We say that $i,j$ are \emph{$l$-equivalent} , denoted by $i\mathop{\sim}\limits_{l}j$, if $e_{i,j}\geq l$ or $i=j$. For any nonempty subset $c$ of $\mathcal{I}$, let $n_{l}(c)$ be the number of $l$-equivalence classes of $c$.

(2) For each subset $c \subseteq\mathcal{I}$ with $|c|\geq 1$, the \emph{level} of $c$ is defined by 
\[L(c):=\min \{e_{i,j}: i,j\in c\}. \]  
\end{defn}
%\begin{defn}

%to be  $l$ such that $n_{l+1}(c)>1$. For each $c=\{i\}$, we define $L(C)=\epsilon_{i}$, for all $i\in\mathcal{I}$.
%\end{defn}

In \cite[Theorem 6.5]{lee2022tate}, Lee proves the following general formula:  
%Let $T_{L/k}$ be the multinorm-one torus associated to a $k$-\'etale algebra $L=\prod_{i=0}^m K_i$. We have

\begin{equation}\label{eq:formula}
  \Sha(L/k) \cong\mathop{\bigoplus}\limits_{r\in\mathcal{R}\setminus\{0\}}\mathbb{Z}/p^{{\Delta_{r}}-r}\mathbb{Z}\mathop{\bigoplus}\limits_{r\in\mathcal{R}}\mathop{\bigoplus}\limits_{l\geq L(U_{r})}\mathop{\bigoplus}\limits_{c\in U_r/\mathop{\sim}\limits_{l}}(\mathbb{Z}/p^{f_{c}-r}\mathbb{Z})^{n_{l+1}(c)-1},  
\end{equation}
where $\mathcal{R}=\{0\leq r\leq\epsilon_0|\text{ } U_{r}\neq\emptyset\}$. We refer to \cite[Sections 4 and 5]{lee2022tate} (also see \cite[Section 2]{HHLY}) for the definitions of the \emph{patching degree} $\Delta_{r}$ of $U_r$ and of the \emph{degree of freedom} $f_c$ of each $l$-equivalence class $c$.

In \cite{HHLY}, Huang, Liang and the present authors investigate the invariants $\Delta_r$ and $f_c$ in Lee's formula when $L=\prod_{i=0}^m K_i$ is assumed to be of Kummer type, namely each cyclic extension $K_i$ is of the form $k(\alpha^{1/p^{\epsilon_i}})$ for some $\alpha\in k^\times$. A basic idea is to describe these invariants in terms of a combinatorial way. % ===%.  % description for these invariants. 
The authors also implemented computer programs for computing the $\Sha(L/k)$ in the following cases:
\begin{itemize}
    \item $k=\Q(\zeta_{p^n})$ is a $p^n$th cyclotomic field extension; \item $F:=k(\ell_1^{1/p^n},\ell_2^{1/p^n})$ 
       is a bicyclic extension over $k$ with distinct rational primes $\ell_1$ and $\ell_2$; and 
     \item each $K_i$ is a cyclic subextension of $F$, that is, $K_i=k(\ell_1^{a_i/p^n}\ell_2^{b_i/p^n})$ for some integers $0\le a_i,b_i<p^n$. 
\end{itemize}
The programs have input data: $p, n, \{a_i,b_i\}_{0\le i\le m}$, and compute several invariants including $\Delta_r$, $c$, $n_l(c)$ and $f_c$ in Lee's formula \eqref{eq:formula}. % for computing the Tate-Shafarevich group $\Sha(L/k)$. 
The programs use the mathematical software SageMath and can be found on
\begin{center}
    \url{https://github.com/hfy880916/Tate-Shafarevich-groups-of-multinorm-one-torus}.
\end{center}

\begin{example}\label{ex:1}
We put $p = 3$ and $n = 3$, so $k = \Q(\zeta_{27})$. Choose the primes $\ell_1 = 5$ and $\ell_2 = 19$. We consider the multinorm-one torus defined by the following extensions over $k$: $K_0 = k(\sqrt[27]{5})$, $K_1 = k(\sqrt[27]{5 \times 19})$, $K_2 = k(\sqrt[27]{5^2 \times 19^3})$, $K_3 = k(\sqrt[27]{5^3 \times 19^5})$, $K_4 = k(\sqrt[27]{5^5 \times 19^{11}})$. We list $a_i$ and $b_i$ as follows:
\begin{align*}
    a_0 &= 1, & a_1 &= 1, & a_2 &= 2, & a_3 &= 3, & a_4 &= 5, \\
    b_0 &= 0, & b_1 &= 1, & b_2 &= 3, & b_3 &= 5, & b_4 &= 11.
\end{align*}
Using Lee's formula \eqref{eq:formula} and the computer program, we compute the Tate-Shafarevich group
\begin{align*}
    \Sha(L/k) \simeq   (\Z/3\Z)^3.
\end{align*}
\end{example}

\begin{example}\label{ex:2}
Let $p, n, k, \ell_1, \ell_2, m$ be the same as in Example~\ref{ex:1}. Consider a different multinorm-one torus defined by the following field extensions: $K_0 = k(\sqrt[27]{5})$, $K_1 = k(\sqrt[27]{5 \times 19})$, $K_2 = k(\sqrt[27]{5^2 \times 19^3})$, $K_3 = k(\sqrt[27]{5^4 \times 19^9})$, $K_4 = k(\sqrt[27]{5^{10}\times 19^{19}})$. We list $a_i$ and $b_i$ as follows:
\begin{align*}
    a_0 &= 1, & a_1 &= 1, & a_2 &= 2, & a_3 &= 4, & a_4 &= 10, \\
    b_0 &= 0, & b_1 &= 1, & b_2 &= 3, & b_3 &= 9, & b_4 &= 19.
\end{align*}
In this case we obtain $\Sha(L/k) \simeq  \Z/3\Z$.
\end{example}

In the first example $K_i$ are linearly disjoint. Our computation result agrees with \cite[Proposition 7.3]{lee2022tate}. In the second example some of $K_i$ are not linearly disjoint so there are some contributions from $U_r$ for $r\ge 1$. We refer the reader to \cite{HHLY} for the details. 

\section*{Acknowledgments}
The present project started from the authors' participation of an undergraduate research program in the National Center for Theoretical Sciences (NCTS). They acknowledge the NCTS for the stimulating environment.
The authors thank Jiangwei Xue for helpful discussions and are grateful to Valentijn Karemaker for her careful reading of an earlier draft and helpful comments.
Yu was partially supported by the NSTC grant 109-2115-M-001-002-MY3 and the Academia Sinica Investigator Award AS-IA-112-M01.

% \makeaddress % To print address at the end. 

\bibliographystyle{plain}
% \bibliography{TeXBiB}

\def\cprime{$'$}

\end{document}